\documentclass[aap,preprint]{imsart}

\RequirePackage[OT1]{fontenc}
\RequirePackage{amsthm,amsmath,amsfonts}
\RequirePackage[numbers]{natbib}
\RequirePackage[colorlinks,citecolor=blue,urlcolor=blue]{hyperref}

\usepackage{blindtext}
\usepackage{enumitem}

\startlocaldefs
\numberwithin{equation}{section}
\theoremstyle{plain}
\newtheorem{theorem}{Theorem}[section]
\newtheorem{lemma}[theorem]{Lemma}
\newtheorem{proposition}[theorem]{Proposition}

\newtheorem{corollary}[theorem]{Corollary}
\theoremstyle{definition}
\newtheorem{definition}[theorem]{Definition}
\newtheorem{remark}[theorem]{Remark}

\endlocaldefs

\allowdisplaybreaks

\begin{document}

\begin{frontmatter}

\title {Existence results for linear evolution equations of parabolic type}
\runtitle{Linear evolution equations of parabolic type}


\begin{aug}
\author{
T\^{o}n Vi$\hat{\d{e}}$t  T\d{a}\ead[label=e1]{tavietton[at]agr.kyushu-u.ac.jp}
}



\runauthor{T$\hat{\rm o}$n V.  T\d{a}}
\affiliation{Kyushu University}

\address{
Center for Promotion of International Education and Research\\
Kyushu University\\
 6-10-1 Hakozaki, Higashi-ku, Fukuoka 812-8581, Japan\\
\printead{e1}   }
\end{aug} 

\begin{abstract}
We study  both strict and mild solutions to  parabolic evolution equations of the form  
$dX+AXdt=F(t)dt+G(t)dW(t)$
in Banach spaces.  
First, we explore the deterministic case. The maximal  regularity of  solutions has been shown. 
Second, we investigate the stochastic case. We prove existence of strict solutions and show their space-time regularity. Finally, we apply our abstract results  to  a  stochastic heat equation.
\end{abstract}


\end{frontmatter}

\section {Introduction}

We consider the Cauchy problem for a linear evolution equation with additive noise
\begin{equation} \label{TVT1}
\begin{cases}
dX+AXdt=F(t)dt+ G(t)dW(t),\hspace{1cm} 0<t\leq T,\\
X(0)=\xi
\end{cases}
\end{equation}
in a    UMD Banach space $E$ of type 2 with  norm $\|\cdot\|$. 
Here, $W$ denotes  a cylindrical Wiener process on a separable Hilbert space $H$, and is defined on a filtered, complete  probability space  $(\Omega, \mathcal F,\{\mathcal F_t\}_{t\geq 0},\mathbb P).$  
 Operators $G(t), 0\leq t\leq T,$ are $\gamma$\,-\,radonifying operators from $H$ to $E$, whereas  $F$ is an $E$\,{-}\,valued measurable function  on  $[0,T]$.  Initial value $ \xi$ is an $E$\,{-}\,valued $\mathcal  F_0$\,{-}\,measurable  random variable.  And,  $A$ is a  sectorial  operator in $E$, i.e. it is a densely defined, closed linear operator satisfying the condition:
\begin{itemize}
  \item [(\rm{A})] The spectrum $\sigma(A)$ of $A$ is  contained in an open sectorial domain $\Sigma_{\varpi}$: 
\begin{equation*} \label{A1} 
\sigma(A) \subset  \Sigma_{\varpi}=\{\lambda \in \mathbb C: |\arg \lambda|<\varpi\}, \quad \quad 0<\varpi<\frac{\pi}{2}.
       \end{equation*}
   The resolvent satisfies the estimate  
\begin{equation*} \label{A2}
          \|(\lambda-A)^{-1}\| \leq \frac{M_{\varpi}}{|\lambda|}, \quad\quad\quad \quad   \lambda \notin \Sigma_{\varpi}
     \end{equation*}
     with some  constant $M_{\varpi}>0$ depending only on the angle $\varpi$.
\end{itemize}

The equation \eqref{TVT1} has been extensively explored in different settings. In the deterministic case (i.e. $G\equiv 0$), it has been investigated by many researchers (see, e.g.,  \cite{Ball,prato-1,Favini,Lunardi,Pazy,Sinestrari},\cite{Tanabe0}-\cite{Tanabe3},\cite{yagi0}-\cite{yagi}). Not only weak solutions but also strict solutions have been studied. The three main approaches are known for this study, namely, the semigroups methods, the variational methods and the methods of using  operational equations.


In the stochastic case (i.e. $G\not\equiv 0$), weak solutions in $L_2$ spaces   have  been  shown in 
\cite{prato0,prato}  by using the semigroup methods, in \cite{Rozovskii} by using the variational methods, and in \cite{Walsh} by using the martingale methods.  After that, some researchers have studied that kind of solutions in weighted Sobolev  spaces or weighted H\"{o}lder  spaces (see  \cite{Krylov0,Krylov1,Mikulevicius1}) by using the semigroup or the variational methods.  

However, existence of strict solutions to \eqref{TVT1} is only shown in a very restrictive case. In \cite{prato0,prato}, Da Prato et al. showed that when $A$ is a bounded linear operator,  \eqref{TVT1} that is considered in Hilbert spaces   possesses strict solutions (under other conditions on coefficients and initial value).


The work in \cite{prato} inspired us to study existence of  strict solutions to \eqref{TVT1}  when the linear operator $A$ is unbounded. 
In the present paper, we want to consider the equation in Banach spaces (for the deterministic case) and in UMD Banach spaces of type 2 (for the stochastic case), where both $F$ and $G$  have  temporal and spatial regularity. 
In the deterministic case, our results improve  those in \cite{Ton1} and generalize the  maximal regularity theorem 
in \cite{yagi}. In the stochastic case, we show existence and regularity of strict solutions, provided that $A$ is a (unbounded) sectorial operator.

Let us assume  that $A^{-\alpha_1}F$ and $A^{-\alpha_2}G \, (-\infty<\alpha_1,\alpha_2<\infty),$ respectively, belong to  weighted H\"{o}lder continuous function spaces $\mathcal F^{\beta, \sigma}((0,T];E)$ and $\mathcal F^{\beta, \sigma}((0,T];\gamma(H;E)).$
 In the deterministic case, the maximal regularity for both initial value $\xi$ and function $F$ are shown  in Theorem \ref{T2}: 
\begin{align*}
(\text{input})   \hspace{0.5cm}&   \xi \in \mathcal D(A^{\beta-\alpha_1}), \quad A^{-\alpha_1}F \in \mathcal F^{\beta, \sigma}((0,T];E).      \\
(\text{output})   \hspace{0.5cm} &  X \in \mathcal C((0,T];\mathcal D(A^{\beta-\alpha_1})), \quad A^{1-\alpha_1}X \in \mathcal F^{\beta, \sigma}((0,T];E), \\
 &  \quad A^{-\alpha_1} \frac{dX}{dt} \in \mathcal  F^{\beta, \sigma}((0,T];E).
\end{align*}
In the stochastic case, \eqref{TVT1} possesses a unique strict solution (see Corollary \ref{cor1}):
\begin{align*}
(\text{input})   \hspace{0.5cm}&   \xi \in \mathcal D(A^{\beta-\alpha_1}), \quad A^{-\alpha_1}F \in \mathcal F^{\beta, \sigma}((0,T];E),      \\
& A^{-\alpha_2}G\in\mathcal F^{\beta, \sigma}((0,T];\gamma(H;E))\quad (\alpha_1\leq 0, \alpha_2 <\alpha_1- \frac{1}{2}). \\
(\text{output})   \hspace{0.5cm} &   \text{There exists a unique strict solution } X \text{ such that } \\
& X \in \mathcal C([0,T];\mathcal D(A^{\beta-\alpha_1})), \quad AX\in \mathcal C ((0,T];E)\hspace{1cm}\text{ a.s.}
\end{align*}

For the study, we use the semigroup methods. In particular, we very often use an identity: 
$$\int_s^t (t-u)^{\alpha-1} (u-s)^{\beta-1} du= (t-s)^{\alpha+\beta-1} B(\beta,\alpha), \hspace{1cm} 0\leq s<t<\infty, $$
where $B(\cdot,\cdot)$ is the Beta function and $0<\alpha, \beta<1$ are some constants. Notice that when $\alpha+\beta=1$, we have
$$\int_s^t (t-u)^{\alpha-1} (u-s)^{-\alpha} du= \frac{\pi}{\sin(\pi \alpha)}, \hspace{1cm} 0\leq s<t<\infty. $$
This identity has been used as a key point in the so-called  {\em factorization method} introduced by  Da Prato et al. (see \cite{prato0,prato}).  

For applications, our results can be applied to  a class of stochastic partial differential equations such as  heat equations, reaction diffusion equations,   FitzHugh-Nagumo models, or Hodgkin-Huxley models (see, e.g., \cite{Walsh,prato}). In the last section of the present paper, we consider a special case, namely 
$\sigma_1=-a(x) u(t,x)+b(t,x)$ and $\sigma_2=\sigma(t,x)$ 
 (see \eqref{TVT53}),  of the nonlinear stochastic heat equation: 
\begin{equation}  \label{TVT52}
\frac{\partial u}{\partial t}=\Delta u+\sigma_1(t,x,u(t,x))+\sigma_2(t,x,u(t,x))\dot W(t,x).
\end{equation} 
We should mention that weak solutions to \eqref{TVT52}  have been studied in  \cite{Mueller,Pardoux3,Shiga,Walsh} and references therein. By using our abstract results,    strict solutions to \eqref{TVT53} can be obtained (see Theorems \ref{T5} and \ref{T6}).



The paper is organized as follows. Section \ref{section2} is preliminary.
 Section \ref{section3} studies  the deterministic case of  \eqref{TVT1}. The  stochastic case is investigated in Section \ref{section4}. Finally, Section \ref{section5} gives an application to heat equations.

 \section{Preliminary} \label{section2}

\subsection {UMD Banach spaces of type 2} 
Let us recall the notion of UMD Banach spaces of type 2.
\begin{definition}  \label{UMD}
\begin{itemize}
  \item [{\rm (i)}]
A Banach space $E$ is called a  UMD  space  if for some (equivalently, for all) $1<p<\infty,$ there is a constant $c_{p}(E)$ such that for any $L^p$\,-\,integrable $E$\,-\,valued martingale difference $\{M_n\}_n$ (i.e. $\{\sum_{i=1}^n M_i\}_{n=1}^\infty $ is a martingale) on a complete probability space $(\Omega',\mathcal F', \mathbb P')$ and  any $\epsilon \colon \{1,2,3,\dots\} \to \{-1,1\},$ 
$$ \mathbb E'\|\sum_{i=1}^n \epsilon (i) M_n\|^p \leq c_{p}(E) \mathbb E' \|\sum_{j=1}^n M_j\|^p, \hspace{2cm}  n=1,2,3,\dots.$$
  \item [{\rm (ii)}]
A Banach space $E$ is said to be of type 2 if there exists  $c_2(E)>0$ such that for any Rademacher sequence $\{\epsilon_i\}_i$ on a complete probability space $(\Omega',\mathcal F', \mathbb P')$ and any finite sequence $\{x_k\}_{k=1}^n$ of $E$,
$$\mathbb E' \|\sum_{i=1}^n \epsilon_i x_i\|^2  \leq c_2(E) \sum_{i=1}^n\|x_i\|^2.$$
(Recall that a Rademacher sequence   is a sequence of independent symmetric random variables each one taking on the set $\{1,-1\}$.) 
\end{itemize}
\end{definition}


\begin{remark}  \label{rm0}
  All Hilbert spaces and  $L^p$ spaces $(2\leq p<\infty) $ are UMD  spaces of type 2.  When  $1<p<\infty$, $L^p$ spaces are UMD  spaces. 
\end{remark}

From now on, if not specified we always assume that $E$ is a UMD Banach sapce of type 2 and $H$ is a separable Hilbert space.

\subsection{$\gamma$\,-\,radonifying operators}  \label{radonifyingOperators}
Let us review the notion of $\gamma$\,-\,radonifying operators. For more details on the subject,  see  \cite{van2}.

\begin{definition}[$\gamma$\,-\,radonifying operators]\label{Def10}
Let $\{e_n\}_{n=1}^\infty$ be  an orthonormal basis of $H$.
Let $\{\gamma_n\}_{n=1}^\infty$ be a sequence of independent standard Gaussian random variables on a probability space  $(\Omega', \mathcal F', \mathbb P')$. A $\gamma$\,-\,radonifying operator from $H$ to  $E$ is an operator, denoted by  $\phi$ for example, in $L(H;E)$ such that  the Gaussian series $\sum_{n=1}^\infty \gamma_n \phi e_n$ converges in $L^2(\Omega',E)$.
\end{definition}

Denote by $\gamma(H;E)$ the set of all $\gamma$\,-\,radonifying operators from $H$ to $E$. Define a   norm in $\gamma(H;E)$ by 
$$\|\phi\|_{\gamma(H;E)}=\Big[\mathbb E'\Big\|\sum_{n=1}^\infty \gamma_n \phi e_n\Big\|^2\Big]^{\frac{1}{2}}, \hspace{2cm} \phi \in \gamma(H;E).$$
It is known  that the norm is independent of the orthonormal basis $\{e_n\}_{n=1}^\infty$ and the Gaussian sequence $\{\gamma_n\}_{n=1}^\infty$. Furthermore,  $(\gamma(H;E), $ $\|\cdot\|_{\gamma(H;E)})$ is complete.

\begin{remark}  \label{rm1}
When $E$ is a Hilbert space, the space $(\gamma(H;E), \|\cdot\|_{\gamma(H;E)})$  is isometrical to the space $L_2(H;E)$ of  Hilbert-Schmidt operators.
\end{remark}

Let $(S,\Sigma)$ be a measurable space. 
 A function $\varphi\colon S\to E$ is said to be strongly measurable if it is the pointwise limit of a sequence of simple functions. 
A  function $\phi\colon S\to L(H;E)$  is said to be $H$\,-\,strongly measurable if $\phi(\cdot)h  \colon S\to E$ is strongly measurable for all $h\in H$.

Denote by $\mathcal N^2([0,T])$
the class of   all $H$\,-\,strongly measurable and adapted processes $\phi\colon [0,T]\times \Omega \to \gamma (H;E)$ in $L^2((0,T)\times \Omega; \gamma (H;E)).$ 

The following result is very often used  in this paper.
\begin{lemma}  \label{T2.12}
Let $\phi_1\in L(E)$ and $\phi_2 \in \gamma(H;E)$. Then, $\phi_1\phi_2 \in \gamma(H;E)$ and
$$\|\phi_1\phi_2\|_{\gamma(H;E)} \leq \|\phi_1\|_{L(E)} \|\phi_2\|_{\gamma(H;E)}.$$
\end{lemma}

\subsection{Stochastic integrals}

\begin{definition}
A family $W=\{W(t)\}_{t\geq 0}$ of bounded linear operators from  $H$ to $L^2(\Omega)$ is called a cylindrical Wiener process on $H$ if
\begin{itemize}
  \item [(i)] $Wh=\{W(t)h\}_{t\geq 0}$  is a scalar Wiener process on  $(\Omega, \mathcal F,\{\mathcal F_t\}_{t\geq 0},\mathbb P)$  for all $h\in H.$
  \item [(ii)] $\mathbb E[W(t_1)h_1 W(t_2)h_2]=\min\{t_1, t_2\} \langle h_1, h_2\rangle_H$, $ 0 \leq t_1, t_2<\infty$, $ h_1, h_2 \in H$.
\end{itemize}  
\end{definition}

For each $\phi \in \mathcal N^2([0,T]),$ the stochastic integral $\int_0^T \phi(t)dW(t)$  is defined as a limit of integrals of adapted step processes. By a localization argument  stochastic integrals can be extended to the class $\mathcal N([0,T])$ of all $H$\,-\,strongly measurable and adapted processes $\phi\colon [0,T]\times \Omega \to \gamma (H;E)$ which are  in $L^2((0,T); \gamma (H;E))$ a.s. (see \cite{van2}).

 \begin{theorem}\label{T9}
There exists  $c(E)>0$ depending only on $E$ such that
$$\mathbb E\Big\|\int_0^T \phi(t)dW(t)\Big\|^2\leq c(E) \|\phi\|_{L^2((0,T)\times \Omega; \gamma (H;E))}^2, \hspace{1cm} \phi \in \mathcal N^2([0,T]),$$
here $\|\phi\|_{L^2((0,T)\times \Omega; \gamma (H;E))}^2=\int_0^T \mathbb E \|\phi(s)\|_{\gamma (H;E)}^2ds.$
In addition, for any  $\phi$ in $ \mathcal N^2([0,T])$  (or $\mathcal N([0,T])),$
    $\{\int_0^t \phi(s)dW(s), 0\leq t\leq T\}$ is an $E$\,-\,valued  continuous  martingale (or local martingale) and  a Gaussian process.
\end{theorem}
For the proof, see e.g.,  \cite{van2}.

\begin{proposition}  \label{T10}
Let $B$ be a closed linear operator on $E$ and $\phi\colon [0,T]\times \Omega\to \gamma (H;E).$
  If both $\phi$ and $B\phi$  belong to $  \mathcal N^2([0,T]), $   then
$$B\int_0^T \phi(t)dW(t)=\int_0^T B\phi(t)dW(t)  \hspace{2cm}  \text{  a.s.}$$
\end{proposition}
The proof for  Proposition \ref{T10} is very similar to one in   \cite{prato}. So, we omit it.

Let us finally restate the Kolmogorov continuity theorem. This theorem gives a sufficient condition for a stochastic process to be H\"{o}lder continuous.

\begin{theorem} \label{Kol1}
Let $\zeta$ be  an $E$\,{-}\,valued stochastic process on $[a,b], 0\leq a <b<\infty$. Assume  that  for some  $c>0,$ $ \epsilon_1>1 $ and $ \epsilon_2>0,$ 
\begin{equation} \label{TVT2}
\mathbb E\|\zeta(t)-\zeta(s)\|^{\epsilon_1}\leq c (t-s)^{1+\epsilon_2}, \hspace{2cm} a\leq s \leq t \leq b.
\end{equation}
Then, $\zeta$ has  a version whose $\mathbb P$\,{-}\,almost all trajectories are H\"{o}lder continuous functions with an  arbitrarily smaller  exponent than $\frac{\epsilon_2}{\epsilon_1}$.
\end{theorem}

When  $\zeta$ is a Gaussian process,  the condition \eqref{TVT2} can be weakened.
\begin{theorem}  \label{Kol2}
Let $\zeta$ be an $E$\,{-}\,valued   Gaussian process on $[a,b], 0\leq a <b<\infty,$ such that $\mathbb E \zeta(t)=0$ for  $ a\leq t \leq b$.  Assume that  for some  $c>0$ and $ 0<\epsilon\leq 1,$
$$
\mathbb E\|\zeta(t)-\zeta(s)\|^2\leq c (t-s)^\epsilon, \hspace{2cm} a\leq s \leq t \leq b.
$$
Then, there exists a modification of $\zeta$ whose $\mathbb P$\,{-}\,almost all trajectories are H\"{o}lder continuous functions with an  arbitrarily smaller  exponent than $\frac{\epsilon}{2}$.
\end{theorem}
For the proofs of Theorems \ref{Kol1} and \ref{Kol2}, see e.g.,   \cite{prato}.

\subsection{Weighted H\"{o}lder continuous function spaces}
For $0<\sigma<\beta \leq 1$, denote by  $\mathcal F^{\beta, \sigma}((0,T];E)$ the space of all $E$\,{-}\,valued continuous functions $f$ on $(0,T]$ (resp. $[0,T]$) when $0<\beta<1$ (resp. $\beta=1$)  with  the properties:
\begin{itemize}
  \item [\rm (i)] When $\beta<1$, 
\begin{equation}  \label{TVT3}
  t^{1-\beta} f(t)  \text{  has a limit as } t\to 0.
  \end{equation}
  \item [\rm (ii)]  $f$ is H\"{o}lder continuous with   exponent $\sigma$ and   weight $s^{1-\beta+\sigma}$, i.e.  
   \begin{equation}   \label{TVT4}
\begin{aligned}
&\sup_{0\leq s<t\leq T} \frac{s^{1-\beta+\sigma}\|f(t)-f(s)\|}{(t-s)^\sigma}\\
&=\sup_{0\leq t\leq T}\sup_{0\leq s<t}\frac{s^{1-\beta+\sigma}\|f(t)-f(s)\|}{(t-s)^\sigma}<\infty.
\end{aligned}
\end{equation}
  \item [\rm (iii)] 
 \begin{equation} \label{TVT5}
  \lim_{t\to 0} w_{f}(t)=0,
  \end{equation}
  where $w_{f}(t)=\sup_{0\leq s  <t}\frac{s^{1-\beta+\sigma}\|f(t)-f(s)\|}{(t-s)^\sigma}$.
\end{itemize}  
It is clear that   $\mathcal F^{\beta, \sigma}((0,T];E)$ is a Banach space with  norm
$$\|f\|_{\mathcal F^{\beta, \sigma}(E)}=\sup_{0\leq t\leq T} t^{1-\beta} \|f(t)\|+ \sup_{0\leq s<t\leq T} \frac{s^{1-\beta+\sigma}\|f(t)-f(s)\|}{(t-s)^\sigma}.$$
 By the definition, for  $ f\in \mathcal F^{\beta, \sigma}((0,T];E),$  
\begin{equation} \label{TVT6}  
\begin{aligned}
\|f(t)\| & \leq   \|f\|_{\mathcal F^{\beta, \sigma}(E)} t^{\beta-1}, \hspace{2cm} 0<t\leq T,\\
 \|f(t)-f(s)\|  & \leq  w_{f}(t) (t-s)^{\sigma} s^{\beta-\sigma-1}\\
& \leq \|f\|_{\mathcal F^{\beta, \sigma}(E)} (t-s)^{\sigma} s^{\beta-\sigma-1},  \hspace{1cm} 0<s\leq t\leq T.
\end{aligned}
\end{equation}

For more details on weighted H\"{o}lder continuous function spaces, see \cite{yagi}.


\subsection{Strict and mild   solutions}
Let us restate the problem  \eqref{TVT1}. Throughout this paper, we  consider \eqref{TVT1} in a UMD Banach space $E$ of type 2, where
\begin{itemize}
  \item [\rm{(i)}] $A$ is a sectorial operator  on $E$.
  \item [\rm{(ii)}]  $W$ is a  cylindrical Wiener process  on a separable Hilbert space $H,$ and  is defined on a complete filtered  probability space  $(\Omega, \mathcal F,\mathcal F_t,\mathbb P).$
   \item [\rm{(iii)}] $F$ is  a measurable   from $[0,T]$ to $(E,\mathcal B(E)).$
  \item  [\rm{(iv)}]  $G\colon [0,T]\to \gamma(H;E)$ such that for any $h\in H$, $G(\cdot)h$ is strongly measurable from $[0,T]$ to $(E,\mathcal B(E)).$ Then, $G\in \mathcal N^2([0,T])$ (see Subsection \ref{radonifyingOperators}).
  \item [\rm{(vi)}] $\xi$ is an $E$-valued $\mathcal F_0$-measurable random variable.
\end{itemize}

\begin{lemma}
Let {\rm (A)}   be satisfied. Then, $(-A)$ generates a  semigroup $S(t)=e^{-tA}$. Furthermore,
\begin{itemize}
                  \item  [\rm (i)]  
 For  $\theta \geq 0$ there exists $\iota_\theta>0$ such that 
\begin{equation} \label{TVT7}
\|A^\theta S(t)\| \leq \iota_\theta t^{-\theta}, \hspace{1cm}  0<t<\infty,
\end{equation}
and 
\begin{equation}  \label{TVT8}
   \|A^{-\theta}\|\leq \iota_\theta.
   \end{equation}
  In particular, 
\begin{equation} \label{TVT9}
\|S(t)\|\leq \iota_0, \hspace{2cm}   0\leq t<\infty.
\end{equation}
\item[\rm (ii)] For  $0< \theta\leq 1, $
 \begin{equation} \label{TVT10}
 t^\theta A^{\theta} S(t)  \text{ converges to } 0 \text{ strongly on } E \text { as } t\to 0.
 \end{equation}
\end{itemize}
\end{lemma}

For the proof, see  e.g.,  \cite{yagi}.

\begin{definition}\label{Def2}
A predictable $E$-valued  process $X$ on $ [0,T]$ is called a strict solution of \eqref{TVT1} if  
$$X(t)\in \mathcal D(A) \quad \text{  and } \quad \Big \|\int_0^t AX(s) ds \Big \|<\infty  \hspace{2cm} \text{a.s., } \, 0<t\leq  T,$$
and  
\begin{align*}
X(t)=&\xi -\int_0^t AX(s)ds+\int_0^t F_1(s)ds+ \int_0^t  G(s)dW(s) \hspace{0.5cm} \text{a.s., } \, 0<t\leq  T.
\end{align*}
\end{definition}

\begin{definition}\label{Def1}
A predictable $E$-valued  process $X$ on $ [0,T]$ is called a  mild solution  of \eqref{TVT1} if 
\begin{align*}
 X(t)=&S(t)\xi +\int_0^tS(t-s) F_1(s)ds\\
&+ \int_0^t S(t-s) G(s)dW(s) \hspace{2cm} \text{a.s., } \, 0<t\leq  T.\notag
\end{align*}
\end{definition}

 A strict (mild) solution $X$ on $[0,T]$ is said to be unique if any other strict (mild) solution $ \bar X$ on $[0,T]$ is indistinguishable from it, i.e. 
$$\mathbb P\{X(t)=\bar X(t) \text { for every } 0\leq t\leq T\}=1.$$

\begin{remark}
A strict solution is a mild solution. The inverse is however not true in general (\cite{prato}).
\end{remark}

\section{The deterministic case}   \label{section3}
In this section, we consider the deterministic case of \eqref{TVT1}, i.e. the equation
\begin{equation} \label{TVT11}
\begin{cases}
dX+AXdt=F(t)dt,\hspace{1cm} 0<t\leq T,\\
X(0)=\xi
\end{cases}
\end{equation}
in a Banach space $E$. (For this case, the UMD and type 2 properties are unnecessary.)

Suppose that 
\begin{itemize}
  \item [(\rm{F1})]  \hspace{0.4cm} 
$ A^{-\alpha_1}F\in \mathcal F^{\beta, \sigma}((0,T];E)  \hspace{0.8cm}  \text{  for some  } 0<\sigma< \beta\leq 1 \text{  and  } -\infty<\alpha_1<1.
$            
\end{itemize}

Let us fist consider the case where the initial value $\xi$ is arbitrary in $E.$
\begin{theorem}\label{T1}  
Let {\rm (A)}  and {\rm (F1)} be satisfied.    
  Then, there exists a unique  mild solution  $X$ to \eqref{TVT11} in the function space:
$$ X\in    \mathcal C((0,T];\mathcal D(A^{1-\alpha_1}))  $$
with the estimate 
\begin{align} 
\|X(t)\|& +t^{1-\alpha_1}\|A^{1-\alpha_1}X(t)\|      \label{TVT12}\\
\leq & C[\|\xi\|+\|A^{-\alpha_1} F\|_{\mathcal F^{\beta,\sigma}(E)}\max\{t^{\beta-\alpha_1}, t^\beta\}], \hspace{1cm} 0<t\leq T.  \notag
\end{align}
   Furthermore, if $\alpha_1\leq 0,$ then $X$ becomes a strict solution of  \eqref{TVT11} possessing the regularity:
$$ X\in  \mathcal C([0,T];E) \cap \mathcal C^1((0,T];E)$$
and satisfying  the estimate  
\begin{equation} \label{TVT13}
t \Big\|\frac{dX}{dt}\Big\|\leq C[\|\xi\|+\|A^{-\alpha_1} F\|_{\mathcal F^{\beta,\sigma}(E)}\max\{t^{\beta-\alpha_1}, t^\beta\}], \hspace{1cm} 0<t\leq T.
\end{equation}
Here, the constant $C$    depends only on the exponents. 
\end{theorem}

\begin{proof}
The proof is divided into four steps.

{\bf Step 1.} 
 Let us show that \eqref{TVT11} possesses a unique mild solution in the space $    \mathcal C((0,T];\mathcal D(A^{1-\alpha_1})).$

 We have
 \begin{align*}
&\int_0^tA^{1-\alpha_1} S(t-s) F(s)ds\\
& =\int_0^tAS(t-s) A^{-\alpha_1}F(s)ds \\
 &=\int_0^tAS(t-s) [A^{-\alpha_1}F(s)-A^{-\alpha_1}F(t)]ds + \int_0^tAS(t-s) dsA^{-\alpha_1}F(t)\\
 &=\int_0^tAS(t-s) [A^{-\alpha_1}F(s)-A^{-\alpha_1}F(t)]ds + [I-S(t)]A^{-\alpha_1}F(t).
 \end{align*}

The integral in the right-hand side of the latter equality is well-defined and continuous on $(0,T].$ This is because by
\eqref{TVT6},  \eqref{TVT7}  and {\rm (F1)},  
\begin{align}
&\int_0^t\|AS(t-s) [A^{-\alpha_1}F(s)-A^{-\alpha_1}F(t)]\|ds  \notag\\
&\leq \int_0^t\|AS(t-s)\| \|A^{-\alpha_1}F(s)-A^{-\alpha_1}F(t)\|ds \notag\\
&\leq \iota_1\|A^{-\alpha_1}F\|_{\mathcal F^{\beta,\sigma}(E)} \int_0^t  (t-s)^{\sigma-1}s^{\beta-\sigma-1}ds \notag\\
& =\iota_1\|A^{-\alpha_1}F\|_{\mathcal F^{\beta,\sigma}(E)} B(\beta-\sigma,\sigma)t^{\beta-1} <\infty, \hspace{1cm} 0<t\leq T,  \label{TVT14}
\end{align}
where $ B(\cdot,\cdot)$ is the Beta function. 
The integral $\int_0^\cdot A^{1-\alpha_1}  S(\cdot-s) F(s)ds$ is  hence  continuous on $(0,T]$.

Since $A^{1-\alpha_1}  $ is closed, we observe that 
$$A^{1-\alpha_1}   \int_0^tS(t-s) F(s)ds=\int_0^tA^{1-\alpha_1}   S(t-s) F(s)ds.$$
Thus,  $A^{1-\alpha_1}   \int_0^\cdot S(\cdot -s) F(s)ds$ is continuous on $(0,T]$.

On the other hand, it is clear that  $A^{1-\alpha_1}  S(\cdot )\xi $ is also continuous on $(0,T]$. The function $X$ defined by 
$$X(t)=A^{\alpha_1-1}  \left[A^{1-\alpha_1}  S(t)\xi +A^{1-\alpha_1} \int_0^tS(t-s) F(s)ds\right]$$
is thus a unique mild solution of \eqref{TVT11} in 
$ \mathcal C((0,T];\mathcal D(A^{1-\alpha_1} )).$

{\bf Step 2.} Let us  verify the estimate \eqref{TVT12}.

  When $\alpha_1< 0,$  \eqref{TVT6},  \eqref{TVT8} and   \eqref{TVT9}  give  
  \begin{align*}
&\int_0^t\|S(t-s) F(s)\|ds  \\ 
& \leq \int_0^t\|A^{\alpha_1} S(t-s)\| \| A^{-\alpha_1}F(s)\|ds      \notag\\
 &\leq \int_0^t \iota_{-\alpha_1} \iota_0 \|A^{-\alpha_1}F\|_{\mathcal F^{\beta,\sigma}(E)} s^{\beta-1} ds    \notag\\
 &=\frac{\iota_{-\alpha_1} \iota_0 \|A^{-\alpha_1}F\|_{\mathcal F^{\beta,\sigma}(E)} t^\beta}{\beta}<\infty, \hspace{1cm} 0\leq t\leq T.  \notag
 \end{align*}
Meanwhile, when $\alpha_1\geq 0$,  \eqref{TVT7} and  \eqref{TVT8} give 
\begin{align*}
&\int_0^t\|S(t-s) F(s)\|ds\\
& \leq \int_0^t\|A^{\alpha_1} S(t-s)\| \| A^{-\alpha_1}F(s)\|ds \\
 &\leq \int_0^t \iota_{\alpha_1}  \|A^{-\alpha_1}F\|_{\mathcal F^{\beta,\sigma}(E)} (t-s)^{-\alpha_1} s^{\beta-1} ds\\
 &=\iota_{\alpha_1}  \|A^{-\alpha_1}F\|_{\mathcal F^{\beta,\sigma}(E)} B(\beta, 1-\alpha_1) t^{\beta-\alpha_1}<\infty, \hspace{0.5cm} 0<t\leq T.
 \end{align*}
Hence, in any case of $\alpha_1$, 
\begin{align}
& \int_0^t\|S(t-s) F(s)\|ds    \label{TVT15}\\
& \leq C \|A^{-\alpha_1}F\|_{\mathcal F^{\beta,\sigma}(E)}  \max\{t^{\beta-\alpha_1}, t^\beta\}, \hspace{1.5cm} 0<t\leq T,   \notag
 \end{align}
where $C$ is some positive constant depending only on the exponents.

Using \eqref{TVT14} and \eqref{TVT15}, we observe that 
\begin{align*} 
&\|X(t)\|+t^{1-\alpha_1}\|A^{1-\alpha_1}X(t)\|      \\
= & \Big\|S(t)\xi+ \int_0^tS(t-s) F(s)ds\Big\|        \notag\\
&+t^{1-\alpha_1}\Big\|A^{1-\alpha_1}S(t)\xi+ \int_0^tA^{1-\alpha_1}S(t-s) F(s)ds\Big\|       \notag\\
\leq & \|S(t)\xi\| + t^{1-\alpha_1}\|A^{1-\alpha_1}S(t)\xi\| +\int_0^t\|S(t-s) F(s)\|ds        \notag\\
&+t^{1-\alpha_1}\int_0^t\|AS(t-s) [A^{-\alpha_1}F(s)-A^{-\alpha_1}F(t)]\|ds         \notag\\
&+t^{1-\alpha_1}\| [I-S(t)]A^{-\alpha_1}F(t)\|       \notag\\
\leq &\|S(t)\xi\| + t^{1-\alpha_1}\|A^{1-\alpha_1}S(t)\| \|\xi\| +C \|A^{-\alpha_1}F\|_{\mathcal F^{\beta,\sigma}(E)}  \max\{t^{\beta-\alpha_1}, t^\beta\}       \notag\\
& +t^{1-\alpha_1}\| I-S(t)\| \|A^{-\alpha_1}F(t)\|       \notag\\
&+\iota_1\|A^{-\alpha_1}F\|_{\mathcal F^{\beta,\sigma}(E)} B(\beta-\sigma,\sigma)t^{\beta-\alpha_1}, \hspace{1cm} 0<t\leq T.        \notag
\end{align*}
Due to \eqref{TVT6},  \eqref{TVT7} and  \eqref{TVT9}, it is then seen that 
\begin{align*} 
&\|X(t)\|+t^{1-\alpha_1}\|A^{1-\alpha_1}X(t)\| \\
\leq &\iota_0 \|\xi\| + \iota_{1-\alpha_1}  \|\xi\| +C \|A^{-\alpha_1}F\|_{\mathcal F^{\beta,\sigma}(E)}  \max\{t^{\beta-\alpha_1}, t^\beta\}\\
& +(1+\iota_0) \|A^{-\alpha_1}F\|_{\mathcal F^{\beta,\sigma}(E)}t^{\beta-\alpha_1}\\
&+\iota_1\|A^{-\alpha_1}F\|_{\mathcal F^{\beta,\sigma}(E)} B(\beta-\sigma,\sigma)t^{\beta-\alpha_1}, \hspace{1cm} 0<t\leq T.
\end{align*}
 Thus, \eqref{TVT12} has been verified.

{\bf Step 3.} 
Let us show that  if $\alpha_1\leq 0,$ then
\begin{itemize}
\item 
$ X\in  \mathcal C([0,T];E) \cap \mathcal C^1((0,T];E).$
\item $X$ is a strict solution of \eqref{TVT11}.
\end{itemize}

In view of \eqref{TVT15}, $\int_0^\cdot S(\cdot-s) F(s)ds$  is continuous on $[0,T]$. Since
$$X(t)=S(t)\xi +\int_0^tS(t-s) F(s)ds,$$ 
we obtain that
\begin{equation} \label{TVT16}
 X\in  \mathcal C([0,T];E).
\end{equation}

 Let $A_n=A(1+\frac{A}{n})^{-1} (n=1,2,3,\dots)$  be the Yosida approximation of $A$. Then, $A_n$ satisfies {\rm (A)}  uniformly and generates an analytic semigroup $S_n(t)$ (see e.g.,  \cite{yagi}). Furthermore, for any $0\leq \nu <\infty$ and $0<t\leq T,$   
\begin{equation} \label{TVT17}
\begin{cases}
\lim_{n\to\infty} A_n^\nu S_n(t)=A^\nu S(t) \hspace{1cm} \text{ in } L(E),  \\
\lim_{n\to\infty} A_n^{-\nu}=A^{-\nu}   \hspace{2cm} \text{ in } L(E),
\end{cases}
\end{equation}
 and   
 \begin{equation} \label{TVT18}
\begin{aligned}
\begin{cases}
\|A_n^\nu S_n(t)\| \leq \varsigma_\nu t^{-\nu} &\quad\quad\text { if   } \nu>0, 0<t\leq T,\\
\|A_n^\nu S_n(t)\| \leq \varsigma_\nu e^{-\varsigma_\nu t} &\quad\quad\text { if  } \nu=0, 0\leq t\leq T,\\
\|A_n^{-\nu}\|  \leq \varsigma_\nu,
\end{cases}
\end{aligned}
\end{equation}
where $\varsigma_\nu>0$ is some constant independent of $n$.

Consider a function $X_n$ defined by 
$$X_n(t)=S_n(t)\xi +\int_0^tA_n^{\alpha_1} S_n(t-s) A^{-\alpha_1}F(s)ds, \hspace{2cm} 0\leq t\leq T.$$
 We have
\begin{align} 
&A_nX_n(t) \notag\\
=&A_nS_n(t)\xi+\int_0^t A_n^{1+\alpha_1}S_n(t-s)[A^{-\alpha_1}F(s)-A^{-\alpha_1}F(t)]ds  \notag \\
&+ \int_0^t A_n^{1+\alpha_1}S_n(t-s)ds A^{-\alpha_1}F(t)\notag\\
=&A_nS_n(t)\xi+\int_0^t A_n^{1+\alpha_1}S_n(t-s)[A^{-\alpha_1}F(s)-A^{-\alpha_1}F(t)]ds    \label{TVT19}\\
&+ A_n^{\alpha_1}[I-S_n(t)]A^{-\alpha_1}F(t). \notag
\end{align}

Let us estimate $\|A_nX_n(t)\|$. When
 $-1\leq \alpha_1 \leq 0$,  
 \eqref{TVT6}, \eqref{TVT18} and  \eqref{TVT19} give 
\begin{align*} 
&\|A_nX_n(t)\|\notag\\
\leq &\varsigma_1 t^{-1} \|\xi\|+\varsigma_{1+\alpha_1}  \|A^{-\alpha_1}F\|_{\mathcal F^{\beta,\sigma}(E)}\int_0^t   (t-s)^{\sigma-\alpha_1-1} s^{\beta-\sigma-1} ds  \notag\\
&+\varsigma_{-\alpha_1} (1+\varsigma_0 e^{-\varsigma_0 t}) \|A^{-\alpha_1}F\|_{\mathcal F^{\beta,\sigma}(E)} t^{\beta-1}\notag\\
=&\varsigma_1 \|\xi\| t^{-1} +  \varsigma_{1+\alpha_1}  \|A^{-\alpha_1}F\|_{\mathcal F^{\beta,\sigma}(E)} B(\beta-\sigma, \sigma-\alpha_1)  t^{\beta-\alpha_1-1}\notag\\
&+\varsigma_{-\alpha_1} (1+\varsigma_0 e^{-\varsigma_0 t}) \|A^{-\alpha_1}F\|_{\mathcal F^{\beta,\sigma}(E)} t^{\beta-1}, \quad \quad 0<t\leq T.
\end{align*}
In the meantime, when $\alpha_1 <-1$, 
\begin{align*} 
&\|A_nX_n(t)\|\notag\\
\leq &\varsigma_1 t^{-1} \|\xi\|+\|A_n^{\alpha_1}\|\int_0^t \|A_n S_n(t-s)\|\|A^{-\alpha_1}F(s)-A^{-\alpha_1}F(t)\|ds \notag\\
&+\varsigma_{-\alpha_1} (1+\varsigma_0 e^{-\varsigma_0 t}) \|A^{-\alpha_1}F\|_{\mathcal F^{\beta,\sigma}(E)} t^{\beta-1}\notag\\
\leq &\varsigma_1 t^{-1} \|\xi\|+\varsigma_{-\alpha_1} \varsigma_1  \|A^{-\alpha_1}F\|_{\mathcal F^{\beta,\sigma}(E)}\int_0^t   (t-s)^{\sigma-1} s^{\beta-\sigma-1} ds  \notag\\
&+\varsigma_{-\alpha_1} (1+\varsigma_0 e^{-\varsigma_0 t}) \|A^{-\alpha_1}F\|_{\mathcal F^{\beta,\sigma}(E)} t^{\beta-1}\notag\\
=&\varsigma_1 \|\xi\| t^{-1} + \varsigma_{-\alpha_1} \varsigma_1   \|A^{-\alpha_1}F\|_{\mathcal F^{\beta,\sigma}(E)} B(\beta-\sigma, \sigma)  t^{\beta-1}\notag\\
&+\varsigma_{-\alpha_1} (1+\varsigma_0 e^{-\varsigma_0 t}) \|A^{-\alpha_1}F\|_{\mathcal F^{\beta,\sigma}(E)} t^{\beta-1}, \hspace{1cm}  0<t\leq T. 
\end{align*}
Therefore, in any case of $\alpha_1$,  there exists $C_1>0$  independent of $n$ such that
\begin{align} 
\|A_nX_n(t)\| \leq & C_1 \|\xi\| t^{-1} + C_1 \|A^{-\alpha_1}F\|_{\mathcal F^{\beta,\sigma}(E)}   \label{TVT20} \\
& \times  \max\{t^{\beta-\alpha_1-1}, t^{\beta-1}\}, \hspace{1cm} 0<t\leq T.   \notag
\end{align}

Thus,  \eqref{TVT17} gives  
$$
\lim_{n\to\infty}A_nX_n(t)=Y(t),
$$
where
\begin{align*}
Y(t)=&AS(t)\xi+\int_0^t A^{1+\alpha_1}S(t-s)[A^{-\alpha_1}F(s)-A^{-\alpha_1}F(t)]ds\\
& +A^{\alpha_1} [I-S(t)]A^{-\alpha_1}F(t).
\end{align*}

Let us verify  that $Y$ is continuous on $(0,T]$. Take $0<t_0\leq T$. By using \eqref{TVT6} and \eqref{TVT7},  for every $t\geq t_0$,
\begin{align*}
&\|Y(t)-Y(t_0)\|\notag\\
=& \Big\| [AS(t)\xi-AS(t_0)\xi] \\
& +\{A^{\alpha_1} [I-S(t)]A^{-\alpha_1}F(t) - A^{\alpha_1} [I-S(t_0)]A^{-\alpha_1}F(t_0)\} \\
&+ \int_{t_0}^t A^{1+\alpha_1}S(t-s)[A^{-\alpha_1}F(s)-A^{-\alpha_1}F(t)]ds\\
&+\int_0^{t_0} A^{1+\alpha_1}S(t-s)[A^{-\alpha_1}F(s)-A^{-\alpha_1}F(t)]ds\\
&-\int_0^{t_0} A^{1+\alpha_1}S(t_0-s)[A^{-\alpha_1}F(s)-A^{-\alpha_1}F(t_0)]ds    \Big\|\\
\leq & \|AS(t_0)[S(t-t_0)-I]\xi\| \notag\\
&+ \|A^{\alpha_1}[I-S(t)]A^{-\alpha_1}F(t)-A^{\alpha_1}[I-S(t_0)]A^{-\alpha_1}F(t_0)\|\notag \\
&+\Big\| \int_{t_0}^t A^{1+\alpha_1}S(t-s)[A^{-\alpha_1}F(s)-A^{-\alpha_1}F(t)]ds\notag\\
&\hspace{0.7cm} +\int_0^{t_0} A^{1+\alpha_1}S(t-s)ds[A^{-\alpha_1}F(t_0)-A^{-\alpha_1}F(t)]\notag\\
&\hspace{0.7cm}+\int_0^{t_0} S(t-t_0)A^{1+\alpha_1}S(t_0-s)[A^{-\alpha_1}F(s)-A^{-\alpha_1}F(t_0)]ds\notag\\
&\hspace{0.7cm}-\int_0^{t_0} A^{1+\alpha_1}S(t_0-s)[A^{-\alpha_1}F(s)-A^{-\alpha_1}F(t_0)]ds\Big\|\notag\\
\leq & \iota_1 t_0^{-1}\|S(t-t_0)\xi-\xi\|\\
& + \|A^{\alpha_1}[I-S(t)]A^{-\alpha_1}F(t)-A^{\alpha_1}[I-S(t_0)]A^{-\alpha_1}F(t_0)\|\notag \\
&+ \int_{t_0}^t \|A^{1+\alpha_1}S(t-s)\| \|A^{-\alpha_1}F(s)-A^{-\alpha_1}F(t)\|ds\notag\\
&+\|A^{\alpha_1}[S(t-t_0)-S(t)][A^{-\alpha_1}F(t_0)-A^{-\alpha_1}F(t)]\|\notag\\
&+\int_0^{t_0} \|[S(t-t_0)-I]A^{1+\alpha_1}S(t_0-s)[A^{-\alpha_1}F(s)-A^{-\alpha_1}F(t_0)]\|ds\notag\\
\leq & \iota_1 t_0^{-1}\|S(t-t_0)\xi-\xi\| \\
&+ \|A^{\alpha_1}[I-S(t)]A^{-\alpha_1}F(t)-A^{\alpha_1}[I-S(t_0)]A^{-\alpha_1}F(t_0)\|\notag \\
&+ \|A^{-\alpha_1}F\|_{\mathcal F^{\beta,\sigma} }\int_{t_0}^t \|A^{1+\alpha_1}S(t-s)\| (t-s)^\sigma s^{\beta-\sigma-1}ds\notag\\
&+\|A^{\alpha_1}[S(t-t_0)-S(t)][A^{-\alpha_1}F(t_0)-A^{-\alpha_1}F(t)]\|\notag\\
&+\|A^{-\alpha_1}F\|_{\mathcal F^{\beta,\sigma} }\|S(t-t_0)-I\|  \int_0^{t_0} \|A^{1+\alpha_1}S(t_0-s)\| (t_0-s)^\sigma s^{\beta-\sigma-1}ds.\notag
\end{align*}
Therefore,
\begin{align*}
&\limsup_{t\searrow  t_0}\|Y(t)-Y(t_0)\|\notag\\
&\leq \limsup_{t\searrow  t_0} \Big[\|A^{-\alpha_1}F\|_{\mathcal F^{\beta,\sigma} }\int_{t_0}^t \|A^{1+\alpha_1}S(t-s)\| (t-s)^\sigma s^{\beta-\sigma-1}ds  \\
&+ \|A^{-\alpha_1}F\|_{\mathcal F^{\beta,\sigma} }\|S(t-t_0)-I\|  \int_0^{t_0} \|A^{1+\alpha_1}S(t_0-s)\| (t_0-s)^\sigma s^{\beta-\sigma-1}ds\Big].
\end{align*}
Thus, it is easily seen that
$$\lim_{t\searrow  t_0}Y(t)=Y(t_0).$$
 Similarly, we obtain that
 $$\lim_{t\nearrow  t_0}Y(t)=Y(t_0).$$ 
The function  $Y$ is hence continuous at $t=t_0$ and then at every point in $(0,T]$.

On the other hand, due to \eqref{TVT17} and \eqref{TVT18}, 
$$\lim_{n\to \infty} X_n(t)=X(t) \hspace{1cm} \text{ in } E \text{  pointwise.}$$
We thus arrive at 
$$X(t)=\lim_{n\to\infty} X_n(t)=\lim_{n\to\infty} A_n^{-1} A_n X_n(t)=A^{-1}Y(t).$$
As a consequence,  
  $$X(t) \in \mathcal D(A),  \hspace{2cm} 0<t\leq T,$$
 and 
$$AX=Y\in \mathcal C((0,T];E).$$

Meanwhile, since $A_n^{\alpha_1}$ is bounded, by some direct calculations,
$$\frac{dX_n}{dt}=-A_n X_n+A_n^{\alpha_1}A^{-\alpha_1}F(t), \hspace{2cm} 0<t\leq T.$$
From this equation, for any $0<\epsilon \leq T,$  
\begin{equation}\label{TVT21}
X_n(t)=X_n(\epsilon)+\int_\epsilon^t[A_n^{\alpha_1}A^{-\alpha_1}F(s)-A_nX_n(s)]ds, \quad\quad \epsilon\leq t\leq T.
\end{equation}

Using  \eqref{TVT20},   the Lebesgue dominate convergence theorem  applied to \eqref{TVT21} provides that 
\begin{equation} \label{TVT22}
X(t)=X(\epsilon)+\int_\epsilon^t[F(s)-AX(s)]ds, \hspace{1cm} \epsilon\leq t \leq T.
\end{equation}
This shows that $X$ is differentiable  on $ [\epsilon,T].$ Since $\epsilon$ is arbitrary in $(0,T]$, we conclude that 
\begin{equation}  \label{TVT23}
X \in \mathcal C^1((0,T];E).
\end{equation}

By combining \eqref{TVT16} and \eqref{TVT23}, the first statement has been verified:
$$X \in \mathcal C([0,T];E) \cap \mathcal C^1((0,T];E).
$$

On the other hand, taking $\epsilon \to 0$ in \eqref{TVT22}, we have
$$X(t)=\xi+\int_0^t[F(s)-AX(s)]ds, \hspace{1cm} 0< t \leq T.$$
Since
\begin{align*}
\int_0^t\|F(s)\|ds \leq & \int_0^t\|A^{\alpha_1}\|\|A^{-\alpha_1}F(s)\| ds\\
&  \leq \|A^{\alpha_1}\|\|A^{-\alpha_1}F\|_{\mathcal F^{\beta,\sigma}(E)} \int_0^t s^{\beta-1} ds<\infty, 
\end{align*}
the integral $\int_0^t F(s)ds$ is well-defined. The latter equality then shows that $\int_0^t AX(s)ds$ is well-defined and that
$$X(t)=\xi-\int_0^tAX(s)ds+ \int_0^t F(s)ds, \hspace{1cm} 0< t \leq T.$$
Therefore,  $X$ is a strict solution of \eqref{TVT11}. We thus arrive at the second statement.

{\bf Step 4.} 
Let us prove  that $X$ satisfies the estimate \eqref{TVT13} when $\alpha_1\leq 0$.

Thanks to \eqref{TVT20}, 
\begin{align*}
&\|AX(t)\|=\|Y(t)\|=\lim_{n\to\infty} \|A_n X_n(t)\|\\
&\leq C_1 \|\xi\| t^{-1} + C_1 \|A^{-\alpha_1}F\|_{\mathcal F^{\beta,\sigma}(E)}  \max\{t^{\beta-\alpha_1-1}, t^{\beta-1}\}, \hspace{1cm} 0<t\leq T.
\end{align*}
Therefore, 
\begin{align*}
&t\|AX(t)\|\leq C_1[ \|\xi\|  +  \|A^{-\alpha_1}F\|_{\mathcal F^{\beta,\sigma}(E)} \max\{t^{\beta-\alpha_1}, t^\beta\}], \hspace{1cm} 0<t\leq T.
\end{align*}

This together with    \eqref{TVT22} gives  
\begin{align*}
\Big\|\frac{dX}{dt}\Big\|=&\|F(t)-AX(t)\|\\
 \leq & \|A^{\alpha_1}\| \|A^{-\alpha_1}F(t)\| +\|AX(t)\|\\
\leq  & \iota_{-\alpha_1}  \|A^{-\alpha_1}F\|_{\mathcal F^{\beta,\sigma}(E)} t^{\beta-1}+C_1 \|\xi\| t^{-1} \\
&+ C_1 \|A^{-\alpha_1}F\|_{\mathcal F^{\beta,\sigma}(E)}  \max\{t^{\beta-\alpha_1-1},t^{\beta-1}\}, \hspace{1cm} 0<t\leq T.
\end{align*}
Hence,   there exists $C>0$ depending only on the exponents and the constants in \eqref{TVT7}, \eqref{TVT8} and \eqref{TVT9} such that 
\begin{align*}
&t\Big\|\frac{dX}{dt}\Big\| \leq C \|\xi\|  + C \|A^{-\alpha_1}F\|_{\mathcal F^{\beta,\sigma}(E)} \max\{t^{\beta-\alpha_1}, t^\beta\}, \hspace{1cm} 0<t\leq T.
\end{align*}

By Steps 1-4, the proof  is  complete.
\end{proof}

Let us now consider the case where the initial value $\xi$ belongs to a subspace of $E$, namely $\mathcal D(A^{\beta-\alpha_1})$.
The below theorem shows   maximal regularity for both initial value $\xi$ and function $F$.

\begin{theorem}\label{T2}  
Let {\rm (A)} and {\rm (F1)} be satisfied.
 Let  $\xi\in \mathcal D(A^{\beta-\alpha_1}).$  Then, there exists a unique  mild solution of \eqref{TVT11} possessing the regularity:
$$X\in \mathcal C((0,T];\mathcal D(A^{1-\alpha_1})) \cap \mathcal C([0,T];\mathcal D(A^{\beta-\alpha_1})),$$
and 
$$A^{1-\alpha_1} X\in \mathcal F^{\beta,\sigma}((0,T];E). $$
  In addition, 
$X$ satisfies the estimate:
\begin{align}
\|A^{\beta-\alpha_1}X\|_{\mathcal C} +\|A^{1-\alpha_1} X\|_{\mathcal F^{\beta,\sigma}(E)} \leq & C[\|A^{\beta-\alpha_1} \xi\|       +\|A^{-\alpha_1} F\|_{\mathcal F^{\beta,\sigma}(E)}]. \label{TVT24} 
\end{align}
Furthermore,
when $  \alpha_1 \leq 0,$   $X$ becomes a strict solution of \eqref{TVT11} possessing the regularity: 
$$X\in \mathcal C^1((0,T];E)$$ 
 and 
$$A^{-\alpha_1}\frac{dX}{dt}\in \mathcal F^{\beta,\sigma}((0,T];E)  $$
with  the estimate:  
\begin{equation} \label{TVT25}
\Big\|A^{-\alpha_1}\frac{dX}{dt}\Big\|_{\mathcal F^{\beta,\sigma}(E)} \leq C [\|A^{\beta-\alpha_1} \xi\| + \|A^{-\alpha_1} F\|_{\mathcal F^{\beta,\sigma}(E)}]. 
\end{equation}
Here,   $C$ is some positive constant depending only on the exponents.
\end{theorem}

\begin{proof}
The proof is divided into several steps.

{\bf Step 1}. Let us verify that 
$$X\in \mathcal C((0,T];\mathcal D(A^{1-\alpha_1})) \cap \mathcal C([0,T];\mathcal D(A^{\beta-\alpha_1})).$$

 In Theorem \ref{T1}, we have already shown that 
$$X\in \mathcal C((0,T];\mathcal D(A^{1-\alpha_1})).$$

We now have
\begin{align*}
A^{\beta-\alpha_1} \int_0^t S(t-s) F(s)ds
= &A^\beta  \int_0^t S(t-s)[A^{-\alpha_1}F(s)-A^{-\alpha_1}F(t)] ds \\
&+ A^{\beta-1} \int_0^t AS(t-s) A^{-\alpha_1}F(t)ds\\
= & \int_0^t A^\beta  S(t-s)[A^{-\alpha_1}F(s)-A^{-\alpha_1}F(t)] ds \\
&+ A^{\beta-1} [I-S(t)] A^{-\alpha_1}F(t).
\end{align*}
Hence,  \eqref{TVT6}, \eqref{TVT7} and {\rm (F1)} give
\begin{align*}
&\|A^{\beta-\alpha_1} \int_0^t S(t-s) F(s)ds\|   \notag\\
 \leq &\int_0^t \|A^\beta  S(t-s)\| \|A^{-\alpha_1}F(s)-A^{-\alpha_1}F(t)\| ds   \notag\\
&+ \|A^{\beta-1} [I-S(t)] A^{-\alpha_1}F(t)\|    \notag\\
 \leq &  \iota_\beta w_{A^{-\alpha_1}F}(t) \int_0^t (t-s)^{\sigma-\beta} s^{\beta-\sigma-1} ds  \notag\\
&+ \|A^{\beta-1} [I-S(t)] A^{-\alpha_1}F(t)\|    \notag\\
= &  \iota_\beta B(\beta-\sigma,1+\sigma-\beta) w_{A^{-\alpha_1}F}(t)   \\  
&+ \|A^{\beta-1} [I-S(t)] A^{-\alpha_1}F(t)\|.     \notag
\end{align*}
In view of   \eqref{TVT5} and \eqref{TVT10}, it follows that
$$\lim_{t\to \infty} \|A^{\beta-\alpha_1} \int_0^t S(t-s) F(s)ds\|=0.$$
The function $A^{\beta-\alpha_1} \int_0^\cdot S(\cdot-s) F(s)ds$ is therefore  continuous at $t=0$.

Since 
$A^{1-\alpha_1} X$ is continuous on $ (0,T]$,  
$A^{\beta-\alpha_1}\int_0^\cdot S(\cdot-s) F(s)ds $ is then continuous on $[0,T]$. 
Thus, from the expression 
$$A^{\beta-\alpha_1} X(\cdot)= S(\cdot) A^{\beta-\alpha_1}  \xi +\int_0^\cdot A^{\beta-\alpha_1}S(\cdot-s) F(s)ds ,$$
we observe that 
$$X\in  \mathcal C([0,T];\mathcal D(A^{\beta-\alpha_1})).$$

In addition, thanks to   \eqref{TVT6}, \eqref{TVT7}, and  \eqref{TVT9}, 
\begin{align}
&\|A^{\beta-\alpha_1}X(t)\|   \notag\\
=&\Big\|S(t) A^{\beta-\alpha_1} \xi+\int_0^t A^{\beta-\alpha_1} S(t-s)F(s)ds\Big\|    \notag \\
\leq  &\iota_0 \|A^{\beta-\alpha_1}\xi\| +\int_0^t \|A^\beta S(t-s)\|\|A^{-\alpha_1}F(s)\|ds    \notag\\
\leq  &\iota_0 \|A^{\beta-\alpha_1}\xi\| + \iota_\beta \|A^{-\alpha_1}F\|_{\mathcal F^{\beta,\sigma}(E)}  \int_0^t (t-s)^{\sigma- \beta }   s^{\beta-\sigma-1}ds    \notag\\
=  &\iota_0 \|A^{\beta-\alpha_1}\xi\| + \iota_\beta B(\beta-\sigma,1+\sigma-\beta)          
\|A^{-\alpha_1}F\|_{\mathcal F^{\beta,\sigma}(E)}, \hspace{0.3cm} 0\leq t \leq T.    \label{TVT26}
\end{align}

{\bf Step 2}. Let us now prove that 
$$A^{1-\alpha_1} X\in \mathcal F^{\beta,\sigma}((0,T];E)$$
and that \eqref{TVT24} holds true.

We  use a decomposition:
\begin{align*}
A^{1-\alpha_1} X(t)
=& A^{1-\alpha_1}S(t) \xi+\int_0^t A S(t-s)[A^{-\alpha_1}F(s)-A^{-\alpha_1}F(t)]ds\\
&+\int_0^t AS(t-s)dsA^{-\alpha_1}F(t)\\
=& A^{1-\alpha_1}S(t) \xi+\int_0^t A^1 S(t-s)[A^{-\alpha_1}F(s)-A^{-\alpha_1}F(t)]ds\\
&+[I-S(t)]A^{-\alpha_1} F(t)\\
=&J_1(t)+J_2(t)+J_3(t).
\end{align*}
Let us show that $J_1, J_2$ and $ J_3$ belong to  $\mathcal F^{\beta,\sigma}((0,T];E)$.

{\it Proof for $J_1$}. Using \eqref{TVT10}  and the expression:  
$$t^{1-\beta} A^{1-\alpha_1}S(t) \xi=t^{1-\beta} A^{1-\beta}S(t) A^{\beta-\alpha_1} \xi,$$
it is easily seen that  
$$\lim_{t\to 0} t^{1-\beta} J_1(t)=0.$$
The condition \eqref{TVT3} is hence fulfilled.

In addition,   \eqref{TVT7} and  \eqref {TVT8} give 
\begin{align}
\sup_{0\leq t \leq T} t^{1-\beta}\|J_1(t)\| 
\leq &  \sup_{t\in[0,T]} t^{1-\beta } \|A^{1-\beta }S(t)\| \| A^{\beta-\alpha_1} \xi\| \notag\\
\leq &  \iota_{1-\beta } \|A^{\beta-\alpha_1} \xi\|.   \label{TVT27}
\end{align}

On the other hand,   for $0<s<t\leq T$,
\begin{align*}
&\frac{s^{1-\beta+\sigma} \|J_1(t)-J_1(s)\|}{(t-s)^\sigma} \notag\\
&=\frac{s^{1-\beta+\sigma} \|A^{1-\alpha_1}[S(t)-S(s)] \xi\|}{(t-s)^\sigma} \notag\\
&\leq \frac{\|A^{-\sigma}[S(t-s)-I]\|}{(t-s)^\sigma}  s^{1-\beta+\sigma}\|A^{1-\beta+\sigma}S(s) A^{\beta-\alpha_1} \xi\|\notag\\
&\leq  \frac{\|\int_0^{t-s} A^{1-\sigma}S(u)du\|}{(t-s)^\sigma} f(s),  
\end{align*}
where 
$$f(s)=s^{1-\beta+\sigma}\|A^{1-\beta+\sigma}S(s) A^{\beta-\alpha_1} \xi\|.$$
Therefore,  \eqref{TVT7} gives 
\begin{align}
\frac{s^{1-\beta+\sigma} \|J_1(t)-J_1(s)\|}{(t-s)^\sigma}
&\leq  \frac{ \int_0^{t-s} \iota_{1-\sigma} u^{\sigma-1}du}{(t-s)^\sigma}f(s)\notag\\
&=  \frac{\iota_{1-\sigma} }{\sigma}f(s)    \notag \\
& \leq  \frac{\iota_{1-\sigma} \iota_{1-\beta+\sigma} \|A^{\beta-\alpha_1} \xi\|}{\sigma}, \hspace{1cm} 0\leq s< t\leq T.    \label{TVT28}
\end{align}
Note  that $f(\cdot)$ is continuous on $[0,T]$ and  
$$\lim_{t\to 0} \sup_{0\leq s \leq t} f(s)=0 \hspace{1cm} \text{  (see }  \eqref{TVT10}).$$
 Thus,
\begin{equation*} 
\sup_{0\leq s<t\leq T}\frac{s^{1-\beta+\sigma} \|J_1(t)-J_1(s)\|}{(t-s)^\sigma}<\infty
\end{equation*}
and 
\begin{align*}
&\lim_{t\to 0} \sup_{0<s<t}\frac{s^{1-\beta+\sigma} \|J_1(t)-J_1(s)\|}{(t-s)^\sigma}=0.
\end{align*}
The conditions \eqref{TVT4} and  \eqref{TVT5} are  then  satisfied. 

We hence conclude that 
$$J_1\in \mathcal F^{\beta,\sigma}((0,T];E).$$ 
Furthermore, thanks to \eqref{TVT27} and \eqref{TVT28},
$$\|J_1\|_{\mathcal F^{\beta,\sigma}(E)} \leq (\iota_{1-\beta }+\frac{\iota_{1-\sigma} \iota_{1-\beta+\sigma} }{\sigma})  \|A^{\beta-\alpha_1} \xi\|.$$

{\it Proof for $J_2$}. The norm of $J_2$ is evaluated by using \eqref{TVT6} and {\rm (F1)}:  
\begin{align*}
\|J_2(t)\|&\leq \int_0^t \|A S(t-s)\|\|A^{-\alpha_1}F(s)-A^{-\alpha_1}F(t)\|ds\notag\\
&\leq \iota_1 w_{A^{-\alpha_1}F}(t)\int_0^t (t-s)^{\sigma-1} s^{\beta-\sigma-1}ds\notag\\
&=\iota_1  B(\beta-\sigma,\sigma) t^{\beta-1 }  w_{A^{-\alpha_1}F}(t), \hspace{1cm} 0<t\leq T. 
\end{align*}
Therefore, 
\begin{align} 
& t^{1-\beta }\|J_2(t)\|  \label{TVT29} \\
&\leq \iota_1  B(\beta-\sigma,\sigma) w_{A^{-\alpha_1}F}(t)     \notag\\
& \leq \iota_1  B(\beta-\sigma,\sigma)\|A^{-\alpha_1}F\|_{\mathcal F^{\beta,\sigma}(E)}, \hspace{0.5cm} 0<t\leq T,   \notag
\end{align} 
and 
\begin{align} 
\lim_{t\to 0} t^{1-\beta } J_2(t)=0.   \label{TVT30}
\end{align}

We now observe that for $0<s<t\leq T,$
\begin{align*}
J_2(t)-J_2(s)=&\int_s^t A S(t-u)[A^{-\alpha_1} F(u)-A^{-\alpha_1} F(t)] du\\
&+ [S(t-s) -I] \int_0^s A S(s-u)[A^{-\alpha_1} F(u)-A^{-\alpha_1} F(s)] du \\
&+\int_0^s A S(t-u)[A^{-\alpha_1} F(s)-A^{-\alpha_1} F(t)] du\\
=&J_{21}(t,s)+J_{22}(t,s)+J_{23}(t,s).
\end{align*}

The norm of  $J_{21}(t,s)$ is estimated  by using  \eqref{TVT6} and \eqref{TVT7}:
\begin{align}
\|J_{21}(t,s)\|\leq & \int_s^t \|A S(t-u)\| \|A^{-\alpha_1} F(u)-A^{-\alpha_1} F(t)\| du   \notag\\
\leq &\int_s^t \iota_1 w_{A^{-\alpha_1} F}(t)(t-u)^{\sigma-1}   u^{\beta-\sigma-1}du\notag\\
\leq &\iota_1 w_{A^{-\alpha_1} F}(t)s^{\beta-\sigma-1}\int_s^t  (t-u)^{\sigma-1}   du\notag\\
=&\frac{\iota_1 w_{A^{-\alpha_1} F}(t)s^{\beta -\sigma-1} (t-s)^\sigma}{\sigma}.   \label{TVT31}
\end{align}

The norm of   $J_{22}(t,s)$ is  evaluated  as follows:
\begin{align*}
&\|J_{22}(t,s)\|\\
&=\Big |\Big|\int_0^{t-s} AS(r)dr  \int_0^s A S(s-u)[A^{-\alpha_1} F(u)-A^{-\alpha_1} F(s)] du \Big |\Big|\\
&=\Big |\Big|\int_0^{t-s}  \int_0^s A^2 S(r+s-u)[A^{-\alpha_1} F(u)-A^{-\alpha_1} F(s)] du dr\Big |\Big|\\
&\leq \iota_2 w_{A^{-\alpha_1} F}(s) \int_0^{t-s}  \int_0^s  (r+s-u)^{-2}   (s-u)^{\sigma}   u^{\beta-\sigma-1}du dr\\
&=\iota_2w_{A^{-\alpha_1} F}(s)    \int_0^s  [(s-u)^{-1}-(t-u)^{-1}]   (s-u)^{\sigma}   u^{\beta-\sigma-1}du \\
&=\iota_2w_{A^{-\alpha_1} F}(s)   (t-s) \int_0^s  (t-u)^{-1}   (s-u)^{\sigma-1}   u^{\beta-\sigma-1}du \\
&= \iota_2w_{A^{-\alpha_1} F}(s)   (t-s) \int_0^s   (t-s+u)^{-1}  u^{\sigma-1}   (s-u)^{\beta-\sigma-1}du.
\end{align*}

Using the decomposition: $\int_0^s=\int_{\frac{s}{2}}^s+\int_0^{\frac{s}{2}}$, we have 
\begin{align}
&\|J_{22}(t,s)\|     \label{TVT32}\\
=& \iota_2w_{A^{-\alpha_1} F}(s)   (t-s)     \int_{\frac{s}{2}}^s   (t-s+u)^{-1}  u^{\sigma-1}   (s-u)^{\beta-\sigma-1}du       \notag\\
&+\iota_2w_{A^{-\alpha_1} F}(s)   (t-s)     \int_0^{\frac{s}{2}}   (t-s+u)^{-1}  u^{\sigma-1}   (s-u)^{\beta-\sigma-1}du.       \notag
\end{align}

In order to handle  the first integral of the latter equality, we have 
\begin{align*}
&(t-s)     \int_{\frac{s}{2}}^s   (t-s+u)^{-1}  u^{\sigma-1}   (s-u)^{\beta-\sigma-1}du\\
=& (t-s)^{\sigma}\int_{\frac{s}{2}}^s   (t-s)^{1-\sigma} (t-s+u)^{-1}  u^\sigma u^{-1}   (s-u)^{\beta-\sigma-1}du\\
\leq & 2(t-s)^{\sigma} s^{-1}\int_{\frac{s}{2}}^s  [ (t-s)^{1-\sigma} (t-s+u)^{-1}  u^\sigma]  (s-u)^{\beta-\sigma-1}du.
\end{align*}
Note that   for every $\frac{s}{2}\leq u\leq s,$
\begin{align*}
(t-s)^{1-\sigma}& (t-s+u)^{-1}  u^\sigma\\
&=\Big(\frac{t-s}{t-s+u}\Big)^{1-\sigma}  \Big(\frac{u}{t-s+u}\Big)^{\sigma}  \leq 1. 
\end{align*}
Hence, 
\begin{align}
&(t-s)     \int_{\frac{s}{2}}^s   (t-s+u)^{-1}  u^{\sigma-1}   (s-u)^{\beta-\sigma-1}du  \notag\\
\leq & 2(t-s)^{\sigma} s^{-1}\int_0^s  (s-u)^{\beta-\sigma-1}du     \notag\\
= & \frac{2(t-s)^{\sigma} s^{\beta -\sigma-1}}{\beta-\sigma}.              \label{TVT33}
\end{align}

In the meantime, 
\begin{align}
& (t-s)     \int_0^{\frac{s}{2}}   (t-s+u)^{-1}  u^{\sigma-1}   (s-u)^{\beta-\sigma-1}du    \notag\\
\leq &2^{1-\beta+\sigma} s^{\beta-\sigma-1} (t-s)   \int_0^{\frac{s}{2}}   (t-s+u)^{-1}  u^{\sigma-1}  du    \notag\\
= &2^{1-\beta+\sigma}     \int_0^{\frac{s}{2(t-s)}}   (1+r)^{-1}  r^{\sigma-1}  dr   s^{\beta-\sigma-1} (t-s)^\sigma     \notag\\
\leq  &2^{1-\beta+\sigma}  \int_0^\infty (1+r)^{-1}  r^{\sigma-1}  dr        s^{\beta -\sigma-1} (t-s)^\sigma.
 \label{TVT34}
\end{align}
(Notice  that  
$\int_0^\infty (1+r)^{-1}  r^{\sigma-1}  dr <\infty.)$

Thanks to the estimates \eqref{TVT32}, \eqref{TVT33},  and \eqref{TVT34},  there exists $C_2>0$ depending only on exponents such that 
\begin{equation} \label{TVT35}
\|J_{22}(t,s)\|\leq C_2        w_{A^{-\alpha_1} F}(s)s^{\beta -\sigma-1}  (t-s)^{\sigma}. 
\end{equation}

The norm of the last term, $J_{23}(t,s)$, is evaluated by using  \eqref{TVT6} and \eqref{TVT9}:
\begin{align}
\|J_{23}(t,s)\|&= \| [S(t-s)-S(t)]  [A^{-\alpha_1} F(s)-A^{-\alpha_1} F(t)]\|      \notag\\
&\leq  \| [S(t-s)-S(t)]\| w_{A^{-\alpha_1} F}(t)s^{\beta-\sigma-1}   (t-s)^{\sigma}        \notag\\
& \leq 2\iota_0 w_{A^{-\alpha_1} F}(t)s^{\beta -\sigma-1} (t-s)^\sigma.   \label{TVT36} 
\end{align}

Thanks to \eqref{TVT29}, \eqref{TVT30},  \eqref{TVT31}, \eqref{TVT35} and \eqref{TVT36}, we conclude that 
\begin{equation*}
J_2\in \mathcal F^{\beta,\sigma}((0,T];E),
\end{equation*}
and 
\begin{equation*}
\|J_2\|_{\mathcal F^{\beta,\sigma}(E)} \leq C_3\|A^{-\alpha_1} F\|_{\mathcal F^{\beta,\sigma}(E)} \hspace{1cm} \text{ with some }  C_3>0.
\end{equation*}

{\it Proof for $J_3$}. Since $t^{1-\beta} A^{-\alpha_1}F(t)$ has a limit as $t\to 0$, 
$$\lim_{t\to 0} t^{1-\beta}J_3(t)=\lim_{t\to 0} [I-S(t)]t^{1-\beta} A^{-\alpha_1}F(t)=0.$$
Furthermore,  \eqref{TVT6}, \eqref{TVT8} and \eqref{TVT9} give 
\begin{align*}
t^{1-\beta}\|J_3(t)\| &\leq  \|I-S(t)\|  t^{1-\beta} \|A^{-\alpha_1}F(t)\|\\
&\leq (1+\iota_0)  \|A^{-\alpha_1}F\|_{\mathcal F^{\beta,\sigma}(E)}, \hspace{1cm} 0\leq t\leq T.
\end{align*}

 We now write
\begin{align}
J_3(t)-J_3(s)=&[I-S(t)][A^{-\alpha_1}F(t)-A^{-\alpha_1}F(s)]   \label{TVT37}\\
&+[I-S(t-s)]S(s)A^{-\alpha_1}F(s).  \notag
\end{align}
The norm of the first term in the right-hand side of the   equality is estimated by using   \eqref{TVT6}, \eqref{TVT8} and \eqref{TVT9}:  
\begin{align}
&\|[I-S(t)][A^{-\alpha_1}F(t)-A^{-\alpha_1}F(s)]\| \notag\\
&\leq \|[I-S(t)]\| w_{A^{-\alpha_1}F}(t) s^{\beta-\sigma-1} (t-s)^\sigma \notag\\
&\leq (1+\iota_0) w_{A^{-\alpha_1}F}(t) s^{\beta -\sigma-1} (t-s)^\sigma \label{TVT38}  \\
&\leq (1+\iota_0) \|A^{-\alpha_1}F\|_{\mathcal F^{\beta,\sigma}(E)} s^{\beta -\sigma-1} (t-s)^\sigma. \notag
\end{align}

Meanwhile, the norm of the second term is evaluated by:
\begin{align*}
&\|[S(t-s)-I]S(s)A^{-\alpha_1}F(s)\|\\
&\leq \|[S(t-s)-I]A^{-\sigma}\|  s^{\beta-\sigma-1} \|s^\sigma A^\sigma  S(s) s^{1-\beta}A^{-\alpha_1}F(s)\|\\
&\leq \Big\|\int_0^{t-s} A^{1-\sigma} S(r) dr\Big\| s^{\beta-\sigma-1} \|  s^\sigma A^\sigma S(s) s^{1-\beta}A^{-\alpha_1}F(s)\|\\
&\leq \int_0^{t-s}  \iota_{1-\sigma} r^{\sigma-1}dr s^{\beta-\sigma-1} \| s^\sigma A^\sigma  S(s) s^{1-\beta}A^{-\alpha_1}F(s)\|\\
&= \frac{\iota_{1-\sigma}}{\sigma}   (t-s)^\sigma s^{\beta-\sigma-1} \| s^\sigma A^\sigma  S(s) s^{1-\beta}A^{-\alpha_1}F(s)\|.
\end{align*}
This means that  there exists $C_4>0$ such that
\begin{align}
&\|A^{\beta-1}[S(t-s)-I]S(s)A^{-\alpha_1}F(s)\| \notag\\
&\leq C_4  (t-s)^\sigma  s^{\beta  -\sigma-1}  \| s^\sigma A^\sigma  S(s) s^{1-\beta}A^{-\alpha_1}F(s)\| \label{TVT39} \\
&\leq C_4  (t-s)^\sigma  s^{\beta  -\sigma-1} s^\sigma \| A^\sigma  S(s)\| s^{1-\beta}\| A^{-\alpha_1}F(s)\|  \notag\\
&\leq C_4  \iota_\sigma \|A^{-\alpha_1}F\|_{\mathcal F^{\beta,\sigma}(E)}s^{\beta  -\sigma-1} (t-s)^\sigma. \notag
\end{align}

In addition, since $ s^{1-\beta}A^{-\alpha_1}F(s)$ has a limit as $s\to 0$,  \eqref{TVT10}  gives  
\begin{equation}   \label{TVT40}
\lim_{s\to 0} \| s^\sigma A^\sigma  S(s) s^{1-\beta}A^{-\alpha_1}F(s)\|=0.
\end{equation}

According to \eqref{TVT37},  \eqref{TVT38}, \eqref{TVT39} and  \eqref{TVT40},  it is seen that 
$$J_3\in\mathcal F^{\beta,\sigma}((0,T];E),$$ 
and   
$$\|J_3\|_{\mathcal F^{\beta,\sigma}(E)} \leq C_5 \|A^{-\alpha_1}F\|_{\mathcal F^{\beta,\sigma}(E)} \hspace{1cm} \text{ with some } C_5\geq 0.$$

We have thus proved that 
$$A^{1-\alpha_1} X \in \mathcal F^{\beta,\sigma}((0,T];E),$$
  and that there exists $C>0$ depending only on the exponents such that 
\begin{align}
\|A^{1-\alpha_1} X\|_{\mathcal F^{\beta,\sigma}(E)} \leq & \sum_{i=1}^3 \|J_i\|_{\mathcal F^{\beta,\sigma}(E)}   \notag \\
 \leq & C[\|A^{\beta-\alpha_1} \xi\|       +\|A^{-\alpha_1} F\|_{\mathcal F^{\beta,\sigma}(E)}]. \label{TVT41}  
\end{align}

The estimate \eqref{TVT24}  then follows from \eqref{TVT26} and \eqref{TVT41}.

{\bf Step 3}. Let us  show the remain of the theorem.

Consider the case $  \alpha_1 \leq 0.$ 
 Thanks to Theorem \ref{T1}, $X$ is a strict solution of \eqref{TVT11} in $ \mathcal C^1((0,T];E)$. 
On  the account of \eqref{TVT22}, it is seen that
\begin{equation*} 
A^{-\alpha_1}\frac{dX}{dt}=A^{-\alpha_1} F(t)- A^{1-\alpha_1}X(t), \hspace{1cm} 0<t\leq T.
\end{equation*}
Since both $A^{-\alpha_1} F$ and $A^{1-\alpha_1} X$ belong to $  \mathcal F^{\beta,\sigma}((0,T];E),$ 
  $$ A^{-\alpha_1}\frac{dX}{dt} \in \mathcal F^{\beta,\sigma}((0,T];E).$$
 In addition,   \eqref{TVT25}  follows from  \eqref{TVT24}  and the estimate:
$$\|A^{-\alpha_1}\frac{dX}{dt}\|_{\mathcal F^{\beta,\sigma}(E)} \leq \|A^{-\alpha_1} F\|_{\mathcal F^{\beta,\sigma}(E)}
 +\|A^{-\alpha_1}X\|_{\mathcal F^{\beta,\sigma}(E)}. $$
By Steps 1-3, the proof is now complete. 
\end{proof}

\begin{remark}
\begin{itemize}
  \item 
Theorem \ref{T2}  improves   Theorem 1 in  \cite{Ton1}. The condition $\frac{1+\sigma}{4} <\alpha_1\leq \frac\beta{2}$ in \cite[Theorem 1]{Ton1} has been removed.
 \item   Theorem \ref{T2} generalizes a result in \cite{yagi}. Indeed,    \cite[Theorem 3.5]{yagi} is a special case of Theorem \ref{T2} (with $\alpha_1=0$).
  \end{itemize}  
\end{remark}

\section{The stochastic case}   \label{section4}

Let us consider the stochastic evolution equation \eqref{TVT1}, where $F$ and $G$ satisfy the following conditions:
\begin{itemize}
  \item [(\rm{F2})]  \hspace{0.4cm} For some  $0<\sigma< \beta-\frac{1}{2}\leq \frac{1}{2} $  and  $-\infty< \alpha_1<1$,
$$ A^{-\alpha_1}F\in \mathcal F^{\beta, \sigma}((0,T];E).$$
  \item [\rm{(G)}]   \hspace{0.4cm} With the $\sigma$ and $\beta$ as above and some $-\infty<\alpha_2<\frac{1}{2}-\sigma$,
$$
  A^{-\alpha_2} G\in \mathcal F^{\beta, \sigma} ((0,T];\gamma(H;E)). $$ 
\end{itemize}

 Throughout this section, the notation $C$  stands for a universal constant which is determined in each occurrence by the exponents. 

Denote by  $W_G$  the stochastic convolution defined by
$$W_G(t)=\int_0^t S(t-s) G(s)dW(s), \hspace{1cm} 0\leq t\leq T. $$
The next two theorems  show 
the regularity of  $W_G$.

\begin{theorem} \label{T3}
Let {\rm (A)} and {\rm (G)} be satisfied.  Let 
 $-\infty < \kappa_1 <\frac{1}{2}-\alpha_2$ and $-\infty<\kappa_2< \min\{\frac{1}{2}-\sigma-\alpha_2,1\}$.
Then,
$$ W_G\in    \mathcal C((0,T];\mathcal D(A^{\kappa_1}))  \hspace{1cm}\text{ a.s.,}$$
$$A^{\kappa_2} W_G\in  \mathcal C^\gamma([\epsilon,T];E) \hspace{1cm}\text{ a.s.,}$$
and 
$$\mathbb E \|A^{\kappa_2} W_G\|\in  \mathcal F^{\beta,\sigma}((0,T];\mathbb R) $$
for any   $  0<\gamma<\sigma $  and $0<\epsilon\leq T$.
In addition, 
\begin{align}
 \mathbb E & \|A^{\kappa_1}W_G(t)\|      \label{TVT43}\\
\leq  &  C\|A^{-\alpha_2}G\|_{\mathcal F^{\beta,\sigma}(\gamma(H;E))}   \max\{t^{\beta-\alpha_2-\kappa_1-\frac{1}{2}}, t^{\beta-\frac{1}{2}}\}, \hspace{1cm} 0<t\leq T,    \notag  
\end{align}
where $C$ is some constant depedning only on the exponents.
Furthermore,  if  $\kappa_1\leq \beta-\alpha_2-\frac{1}{2},$ then
$$ W_G\in    \mathcal C([0,T];\mathcal D(A^{\kappa_1}))  \hspace{1cm}\text{ a.s.}$$
\end{theorem}

\begin{proof}
 We divide the proof into three  steps.

{\bf Step 1}.   Let us show that   
 \begin{itemize}
  \item 
 $W_G\in \mathcal C((0,T];\mathcal D(A^{\kappa_1})) $  a.s.
\item  $W_G$ satisfies \eqref{TVT43}. 
\item  $W_G\in \mathcal C([0,T];\mathcal D(A^{\kappa_1}) $  a.s. when  $\kappa_1\leq \beta-\frac{1}{2}.$
\end{itemize}

We have
\begin{align*}
&\int_0^t\|A^{\kappa_1} S(t-s)G(s)\|_{\gamma(H;E)}^2 ds\\
&\leq \int_0^t\| A^{\alpha_2+\kappa_1}  S(t-s)\|^2 \|A^{-\alpha_2}G(s)\|_{\gamma(H;E)}^2 ds\\
&\leq  \|A^{-\alpha_2}G\|_{\mathcal F^{\beta,\sigma}(\gamma(H;E))}^2   \int_0^t\| A^{\alpha_2+\kappa_1}  S(t-s)\|^2 s^{2(\beta-1)} ds.
\end{align*}
If $\alpha_2+\kappa_1\geq 0,$ then 
\begin{align}
&\int_0^t\|A^{\kappa_1} S(t-s)G(s)\|_{\gamma(H;E)}^2 ds  \notag\\
\leq &  \iota_{\alpha_2+\kappa_1}^2   \|A^{-\alpha_2}G\|_{\mathcal F^{\beta,\sigma}(\gamma(H;E))}^2\int_0^t  (t-s)^{-2(\alpha_2+\kappa_1)} s^{2(\beta-1)}ds\notag\\
=&  \iota_{\alpha_2+\kappa_1}^2   \|A^{-\alpha_2}G\|_{\mathcal F^{\beta,\sigma}(\gamma(H;E))}^2 B(2\beta-1,1-2\alpha_2-2\kappa_1)    \label{TVT44}\\
& \times t^{2(\beta-\alpha_2-\kappa_1)-1}  \notag \\
< & \infty, \hspace{3cm}  0<t\leq T.  \notag
\end{align}
Meanwhile, if $\alpha_2+\kappa_1< 0,$ then 
\begin{align}
&\int_0^t\|A^{\kappa_1} S(t-s)G(s)\|_{\gamma(H;E)}^2 ds  \notag\\
& \leq  C   \|A^{-\alpha_2}G\|_{\mathcal F^{\beta,\sigma}(\gamma(H;E))}^2\int_0^t  s^{2(\beta-1)}ds  \notag\\
&\leq C   \|A^{-\alpha_2}G\|_{\mathcal F^{\beta,\sigma}(\gamma(H;E))}^2 t^{2\beta-1}<\infty, \hspace{2cm}  0\leq t\leq T.  \label{TVT45}
\end{align}
Therefore, $\int_0^\cdot A^{\kappa_1} S(\cdot-s)G(s)dW(s)$ is well-defined and  continuous on $(0,T]$. Since $A^{\kappa_1}$ is closed, we obtain that 
$$A^{\kappa_1} W_G(t)=\int_0^tA^{\kappa_1} S(t-s)G(s)dW(s).$$
 Thus,
$A^{\kappa_1} W_G$ is continuous on $(0,T]$, i.e.
$$W_G\in \mathcal C((0,T];\mathcal D(A^{\kappa_1})) \hspace{1cm} \text{ a.s.}$$

In addition,   \eqref{TVT44} and \eqref{TVT45} give  
\begin{align*}
 \mathbb E \|A^{\kappa_1}W_G(t)\|
%
\leq & \sqrt{ \mathbb E\Big\| \int_0^t A^{\kappa_1}S(t-s)G(s)dW(s) \Big\|^2}\\
%
\leq  & \sqrt{c(E) \int_0^t \|A^{\kappa_1}S(t-s)G(s)\|^2ds }\\
%
\leq &C\|A^{-\alpha_2}G\|_{\mathcal F^{\beta,\sigma}(\gamma(H;E))}   \max\{t^{\beta-\alpha_2-\kappa_1-\frac{1}{2}}, t^{\beta-\frac{1}{2}}\}.
\end{align*}
The estimate  \eqref{TVT43} therefore has been proved. 
 
Furthermore, when $\kappa_1\leq \beta-\alpha_2-\frac{1}{2},$  \eqref{TVT44} also holds true at $t=0$. Thus.  $A^{\kappa_1}W_G$ is also continuous at $t=0$.

{\bf Step 2}.   
Let us verify that there exists  an increasing function  $m(\cdot)$ defined on $(0,T]$ such that 
$$\lim_{t\to 0} m(t)=0$$
 and 
\begin{align*}
\mathbb E\|A^{\kappa_2} W_G(t)-A^{\kappa_2} W_G(s)\|^2 \leq & m(t)^2 s^{2(\beta-\sigma-1)} (t-s)^{2\sigma}, \hspace{0.5cm} 0< s\leq t\leq T.
\end{align*}

From the expression
\begin{align*}
A^{\kappa_2} W_G(t)=&\int_0^t A^{\kappa_2} S(t-r)[G(r)-G(t)]dW(r)\\
&+\int_0^t A^{\kappa_2}  S(t-r) G(t)dW(r),
\end{align*}
it is seen that 
\begin{align*}
&A^{\kappa_2} W_G(t)-A^{\kappa_2} W_G(s)\\
=&\int_s^t A^{\kappa_2} S(t-r)[G(r)-G(t)]dW(r)\\
&+\int_0^sA^{\kappa_2} S(t-r)[G(r)-G(t)]dW(r)\\
&-\int_0^s A^{\kappa_2} S(s-r)[G(r)-G(s)]dW(r)\\
&+\int_s^t A^{\kappa_2} S(t-r)G(t)dW(r)+\int_0^s A^{\kappa_2} S(t-r)G(t)dW(r)\\
&-\int_0^s A^{\kappa_2} S(s-r)G(s)dW(r)\\
=&\int_s^t A^{\kappa_2} S(t-r)[G(r)-G(t)]dW(r)\\
&+\int_0^s A^{\kappa_2} S(t-s)S(s-r)[G(r)-G(s)+G(s)-G(t)]dW(r)\\
&-\int_0^s A^{\kappa_2} S(s-r)[G(r)-G(s)]dW(r)+\int_s^t A^{\kappa_2} S(t-r)G(t)dW(r)\\
&+\int_0^s A^{\kappa_2} S(t-r)G(t)dW(r) \\
&+\int_0^s A^{\kappa_2} S(s-r)[G(t)-G(s)-G(t)]dW(r)\\
=&\int_s^t A^{\kappa_2} S(t-r)[G(r)-G(t)]dW(r)\\
&+\int_0^s [S(t-s)-I]A^{\kappa_2} S(s-r)[G(r)-G(s)]dW(r)\\
&+\int_0^s A^{\kappa_2} S(t-r)[G(s)-G(t)]dW(r)+\int_s^t A^{\kappa_2} S(t-r)G(t)dW(r)\\
&+\int_0^s A^{\kappa_2} S(s-r)[G(t)-G(s)]dW(r)\\
&+\int_0^s A^{\kappa_2} [S(t-r)-S(s-r)]G(t)dW(r)\\
=&K_1+K_2+K_3+K_4+K_5+K_6.
\end{align*}

Let us give estimates for   $\mathbb E\|K_i\|^2 (i=1,\dots,6)$. For $\mathbb E\|K_1\|^2$,   \eqref{TVT6} gives 
\begin{align*}
\mathbb E\|K_1\|^2 \leq  &c(E) \int_s^t \|A^{\kappa_2} S(t-r)[G(r)-G(t)]\|_{\gamma(H;E)}^2dr\\
\leq &c(E) \int_s^t \|A^{\alpha_2+\kappa_2} S(t-r)\|^2 \|A^{-\alpha_2}G(r)-A^{-\alpha_2}G(t)\|_{\gamma(H;E)}^2dr\\
\leq & c(E) w_{A^{-\alpha_2}G}(t)^2 \int_s^t \|A^{\alpha_2+\kappa_2} S(t-r)\|^2 (t-r)^{2\sigma} r^{2(\beta-\sigma-1)}dr,
\end{align*}
where 
$$w_{A^{-\alpha_2}G}(t)=\sup_{0\leq s< t}\frac{s^{1-\beta+\sigma}\|A^{-\alpha_2}G(t)-A^{-\alpha_2}G(s)\|_{\gamma(H;E)}}{(t-s)^\sigma}.$$
If  $\alpha_2+\kappa_2\geq 0,$ then by \eqref{TVT7}, 
\begin{align*}
\mathbb E\|K_1\|^2  \leq &c(E) \iota_{\alpha_2+\kappa_2}^2 w_{A^{-\alpha_2}G}(t)^2  \int_s^t (t-r)^{2(\sigma-\alpha_2-\kappa_2)} r^{2(\beta-\sigma-1)}dr\\
\leq &c(E) \iota_{\alpha_2+\kappa_2}^2 w_{A^{-\alpha_2}G}(t)^2  s^{2(\beta-\sigma-1)}  \int_s^t (t-r)^{2(\sigma-\alpha_2-\kappa_2)}dr\\
= &c(E) \iota_{\alpha_2+\kappa_2}^2  w_{A^{-\alpha_2}G}(t)^2  s^{2(\beta-\sigma-1)}   \frac{(t-s)^{1+2(\sigma-\alpha_2-\kappa_2)}}{1+2(\sigma-\alpha_2-\kappa_2)}\\
\leq &C   w_{A^{-\alpha_2}G}(t)^2    t^{1-2(\alpha_2+\kappa_2)}    s^{2(\beta-\sigma-1)}(t-s)^{2\sigma}.
\end{align*}
If  $\alpha_2+\kappa_2< 0,$ then by \eqref{TVT8} and \eqref{TVT9},
\begin{align*}
\mathbb E\|K_1\|^2  \leq & C   w_{A^{-\alpha_2}G}(t)^2  \int_s^t (t-r)^{2\sigma} r^{2(\beta-\sigma-1)}dr\\
\leq & C   w_{A^{-\alpha_2}G}(t)^2   s^{2(\beta-\sigma-1)}  \int_s^t (t-r)^{2\sigma}dr\\
=& \frac{C }{1+2\sigma}  w_{A^{-\alpha_2}G}(t)^2   s^{2(\beta-\sigma-1)}  (t-s)^{1+2\sigma}\\
\leq & C   w_{A^{-\alpha_2}G}(t)^2   s^{2(\beta-\sigma-1)}  (t-s)^{2\sigma}.
\end{align*}
Hence,
$$\mathbb E\|K_1\|^2  \leq C   w_{A^{-\alpha_2}G}(t)^2  \max\{t^{1-2(\alpha_2+\kappa_2)},t\} s^{2(\beta-\sigma-1)}  (t-s)^{2\sigma}.$$

For $\mathbb E\|K_2\|^2$  we have
\begin{align*}
&\mathbb E\|K_2\|^2 \\
\leq  & c(E) \int_0^s \Big\|\int_0^{t-s} AS(\rho) d\rho A^{\kappa_2}  S(s-r) [G(r)-G(s)]\Big\|_{\gamma(H;E)}^2 dr\\
 \leq &  c(E) \Big\|\int_0^{t-s} A^{1-\sigma} S(\rho) d\rho\Big\|^2  \int_0^s  \|A^{\alpha_2+\kappa_2+\sigma} S(s-r)\|^2 \\
 &\times\|A^{-\alpha_2}G(r)-A^{-\alpha_2}G(s)\|_{\gamma(H;E)}^2 dr\\
 \leq & C w_{A^{-\alpha_2}G}(s)^2 \Big(\int_0^{t-s} \rho^{-1+\sigma} d\rho\Big)^2  \int_0^s    r^{2(\beta-\sigma-1)} (s-r)^{2\sigma}  \\
 &\times \|A^{\alpha_2+\kappa_2+\sigma} S(s-r)\|^2  dr\\
%
%
 \leq & C w_{A^{-\alpha_2}G}(s)^2  (t-s)^{2\sigma}  \int_0^s    r^{2(\beta-\sigma-1)} (s-r)^{2\sigma}  \|A^{\alpha_2+\kappa_2+\sigma} S(s-r)\|^2  dr.
 \end{align*}
If $\alpha_2+\kappa_2+\sigma\geq 0,$ then
 \begin{align*}
&\mathbb E\|K_2\|^2   \\
\leq  &C w_{A^{-\alpha_2}G}(s)^2 (t-s)^{2\sigma} \int_0^s   (s-r)^{-2(\alpha_2+\kappa_2)}   r^{2(\beta-\sigma-1)}dr\\
= & C  B(2\beta-2\sigma-1,1-2\alpha_2-2\kappa_2)w_{A^{-\alpha_2}G}(s)^2 s^{2(\beta-\sigma-\alpha_2-\kappa_2)-1}  (t-s)^{2\sigma}\\
\leq & C w_{A^{-\alpha_2}G}(s)^2  s^{1-2(\alpha_2+\kappa_2)}    s^{2(\beta-\sigma-1)} (t-s)^{2\sigma}.
\end{align*}
If $\alpha_2+\kappa_2+\sigma< 0,$ then
\begin{align*}
\mathbb E\|K_2\|^2 \leq  &C w_{A^{-\alpha_2}G}(s)^2 (t-s)^{2\sigma} \int_0^s   (s-r)^{2\sigma}   r^{2(\beta-\sigma-1)}dr\\
= & C  B(2\beta-2\sigma-1,1+2\sigma)w_{A^{-\alpha_2}G}(s)^2 s^{2\beta-1}  (t-s)^{2\sigma}\\
\leq & C w_{A^{-\alpha_2}G}(s)^2 s^{1+2\sigma}  s^{2(\beta-\sigma-1)} (t-s)^{2\sigma}.
\end{align*}
Hence,
$$\mathbb E\|K_2\|^2  \leq C w_{A^{-\alpha_2}G}(s)^2 \max\{ s^{1-2(\alpha_2+\kappa_2)}, s^{1+2\sigma}\} s^{2(\beta-\sigma-1)} (t-s)^{2\sigma}.$$

For $\mathbb E\|K_3\|^2$, 
\begin{align*}
\mathbb E\|K_3\|^2 \leq  & c(E) \int_0^s \|A^{\alpha_2+\kappa_2} S(t-r) [A^{-\alpha_2}G(s)-A^{-\alpha_2}G(t)]\|_{\gamma(H;E)}^2dr\\
\leq &  c(E) w_{A^{-\alpha_2}G}(t)^2    s^{2(\beta-\sigma-1)} (t-s)^{2\sigma} \int_0^s \|A^{\alpha_2+\kappa_2} S(t-r)\|^2  dr.
\end{align*}
If $\alpha_2+\kappa_2\geq 0,$ then
\begin{align*}
\mathbb E\|K_3\|^2  \leq & C \int_0^s (t-r)^{-2(\alpha_2+\kappa_2)}dr  w_{A^{-\alpha_2}G}(t)^2   s^{2(\beta-\sigma-1)} (t-s)^{2\sigma}\\
 \leq & C  [t^{1-2(\alpha_2+\kappa_2)}-(t-s)^{1-2(\alpha_2+\kappa_2)} ] w_{A^{-\alpha_2}G}(t)^2  s^{2(\beta-\sigma-1)} (t-s)^{2\sigma}\\
 \leq & C w_{A^{-\alpha_2}G}(t)^2 t^{1-2(\alpha_2+\kappa_2)} s^{2(\beta-\sigma-1)}   (t-s)^{2\sigma},
\end{align*}
whereas if $\alpha_2+\kappa_2< 0,$ then
\begin{align*}
\mathbb E\|K_3\|^2  \leq &  \int_0^s C dr  w_{A^{-\alpha_2}G}(t)^2  s^{2(\beta-\sigma-1)} (t-s)^{2\sigma}\\
 \leq & C w_{A^{-\alpha_2}G}(t)^2  s^{2(\beta-\sigma-1)}   (t-s)^{2\sigma}.
\end{align*}
Hence,
$$\mathbb E\|K_3\|^2\leq  C w_{A^{-\alpha_2}G}(t)^2 \max\{t^{1-2(\alpha_2+\kappa_2)}, 1\} s^{2(\beta-\sigma-1)}   (t-s)^{2\sigma}.$$

For $\mathbb E\|K_4\|^2$, 
\begin{align*}
\mathbb E\|K_4\|^2\leq  & c(E) \int_s^t \|A^{\kappa_2} S(t-r) G(t)\|_{\gamma(H;E)}^2dr\\
\leq & c(E)  \int_s^t \|A^{\alpha_2+\kappa_2} S(t-r)\|^2 \|A^{-\alpha_2}G(t)\|_{\gamma(H;E)}^2dr\\
\leq &  c(E)  \|A^{-\alpha_2}G\|_{\mathcal F^{\beta,\sigma}(\gamma(H;E)}^2  t^{2(\beta-1)}  \int_s^t \|A^{\alpha_2+\kappa_2} S(t-r)\|^2 dr\\
\leq &  c(E)  \|A^{-\alpha_2}G\|_{\mathcal F^{\beta,\sigma}(\gamma(H;E)}^2  t^{2\sigma} s^{2(\beta-\sigma-1)}  \int_s^t \|A^{\alpha_2+\kappa_2} S(t-r)\|^2 dr,
\end{align*}
here we used the inequality 
$$t^{2(\beta-1)}=t^{2\sigma} t^{2(\beta-\sigma-1)} \leq  t^{2\sigma} s^{2(\beta-\sigma-1)}.$$

The integral $\int_s^t \|A^{\alpha_2+\kappa_2} S(t-r)\|^2 dr$ can be estimated 
similarly to the integral $ \int_0^s \|A^{\alpha_2+\kappa_2} S(t-r)\|^2  dr$ in the estimate for $\mathbb E\|K_3\|^2$. Thereby,   
\begin{align*}
 &\int_s^t \|A^{\alpha_2+\kappa_2} S(t-r)\|^2  dr \\
& \leq C \max\{ (t-s)^{1-2(\alpha_2+\kappa_2)}, t-s \}  \\
& =C \max\{ (t-s)^{1-2(\alpha_2+\kappa_2+\sigma)}, (t-s)^{1-2\sigma} \} (t-s)^{2\sigma}  \\
&\leq  C \max\{ (t^{1-2(\alpha_2+\kappa_2+\sigma)}, t^{1-2\sigma} \} (t-s)^{2\sigma}.
\end{align*}
Therefore,
\begin{align*}
\mathbb E\|K_4\|^2
\leq  & C \|A^{-\alpha_2}G\|_{\mathcal F^{\beta,\sigma}(\gamma(H;E)}^2   \max\{ (t^{1-2(\alpha_2+\kappa_2)}, t \}   s^{2(\beta-\sigma-1)}  (t-s)^{2\sigma}.
\end{align*}

For $\mathbb E\|K_5\|^2$, 
\begin{align*}
\mathbb E\|K_5\|^2 \leq & c(E) \int_0^s \|A^{\alpha_2+\kappa_2} S(s-r)[A^{-\alpha_2} G(t)-A^{-\alpha_2} G(s)]\|_{\gamma(H;E)}^2dr\\
\leq & c(E)  w_{A^{-\alpha_2}G}(t)^2 s^{2(\beta-\sigma-1)}  (t-s)^{2\sigma} \int_0^s \|A^{\alpha_2+\kappa_2} S(s-r)\|^2 dr.
\end{align*}
By considering two cases: $\alpha_2+\kappa_2 \geq 0$ and $\alpha_2+\kappa_2 < 0$ as for $\mathbb E\|K_3\|^2$ and $\mathbb E\|K_4\|^2$, we arrive at 
\begin{align*}
\mathbb E\|K_5\|^2 
 \leq C w_{A^{-\alpha_2}G}(t)^2  \max\{s^{1-2(\alpha_2+\kappa_2)},s \} s^{2(\beta-\sigma-1)}   (t-s)^{2\sigma}.
\end{align*}

Finally, for $\mathbb E\|K_6\|^2$ we have
\begin{align*}
&\mathbb E\|K_6\|^2\\
\leq & c(E)  \int_0^s \|A^{\kappa_2} [S(t-r)-S(s-r)]G(t)\|_{\gamma(H;E)}^2 dr\\
\leq &  c(E)  \int_0^s \|A^{\alpha_2+\kappa_2+\sigma}S(s-r)\|^2 \|S(t-s)-I]A^{-\sigma}\|^2 \|A^{-\alpha_2}G(t)\|_{\gamma(H;E)}^2 dr\\
= &  c(E)  \int_0^s \|A^{\alpha_2+\kappa_2+\sigma}S(s-r)\|^2 dr \Big\|\int_0^{t-s} A^{1-\sigma} S(\rho)d\rho\Big\|^2  \|A^{-\alpha_2}G(t)\|_{\gamma(H;E)}^2\\
\leq  &  c(E)  \int_0^s \|A^{\alpha_2+\kappa_2+\sigma}S(s-r)\|^2 dr \Big(\int_0^{t-s} \rho^{-1+\sigma} d\rho\Big)^2 \notag\\
& \times \|A^{-\alpha_2}G\|_{\mathcal F^{\beta,\sigma}(\gamma(H;E)}^2 t^{2(\beta-1)}\\
\leq & C  \int_0^s \|A^{\alpha_2+\kappa_2+\sigma}S(s-r)\|^2 dr  \|A^{-\alpha_2}G\|_{\mathcal F^{\beta,\sigma}(\gamma(H;E)}^2 t^{2(\beta-1)} (t-s)^{2\sigma}.
\end{align*}

If $\alpha_2+\kappa_2+\sigma\geq 0,$ then
\begin{align*}
\mathbb E\|K_6\|^2 \leq & C   \int_0^s (s-r)^{-2(\alpha_2+\kappa_2+\sigma)} dr  \|A^{-\alpha_2}G\|_{\mathcal F^{\beta,\sigma}(\gamma(H;E))}^2 t^{2(\beta-1)} (t-s)^{2\sigma}\\
\leq &C  \|A^{-\alpha_2}G\|_{\mathcal F^{\beta,\sigma}(\gamma(H;E))}^2   s^{1-2(\alpha_2+\kappa_2+\sigma)}t^{2\sigma} t^{2(\beta-\sigma-1)} (t-s)^{2\sigma}\\
\leq & C   \|A^{-\alpha_2}G\|_{\mathcal F^{\beta,\sigma}(\gamma(H;E))}^2  t^{1-2(\alpha_2+\kappa_2)} s^{2(\beta-\sigma-1)} (t-s)^{2\sigma}.
\end{align*}
If  $\alpha_2+\kappa_2+\sigma< 0,$ then
\begin{align*}
 \mathbb E\|K_6\|^2   
\leq &    \int_0^s Cdr  \|A^{-\alpha_2}G\|_{\mathcal F^{\beta,\sigma}(\gamma(H;E))}^2 t^{2(\beta-1)} (t-s)^{2\sigma}\\
= &C s  \|A^{-\alpha_2}G\|_{\mathcal F^{\beta,\sigma}(\gamma(H;E))}^2t^{2\sigma} t^{2(\beta-\sigma-1)} (t-s)^{2\sigma}\\
\leq & C\|A^{-\alpha_2}G\|_{\mathcal F^{\beta,\sigma}(\gamma(H;E))}^2  t^{1+2\sigma} s^{2(\beta-\sigma-1)} (t-s)^{2\sigma}.
\end{align*}
Therefore,
$$\mathbb E\|K_6\|^2 \leq  C \|A^{-\alpha_2}G\|_{\mathcal F^{\beta,\sigma}(\gamma(H;E))}^2  \max\{t^{1-2(\alpha_2+\kappa_2)},  t^{1+2\sigma}\}                   s^{2(\beta-\sigma-1)} (t-s)^{2\sigma}.$$

In this way, we conclude that 
\begin{align*}
\mathbb E\|A^{\kappa_2} W_G(t)-A^{\kappa_2} W_G(s)\|^2  \leq & 6\sum_{i=1}^6 \mathbb E\|K_i\|^2 \\
 \leq & m(t)^2 s^{2(\beta-\sigma-1)} (t-s)^{2\sigma},  
\end{align*}
where $m(\cdot)$ is some increasing function  defined on $(0,T]$ such that 
$$\lim_{t\to 0} m(t)=0.$$

{\bf Step 3}.  Let us  verify that  for any $  0<\gamma<\sigma$  and $0<\epsilon\leq T,$
$$A^{\kappa_2} W_G\in  \mathcal C^\gamma([\epsilon,T];E) \hspace{1cm}\text{ a.s.},$$
and 
$$\mathbb E \|A^{\kappa_2} W_G\|\in  \mathcal F^{\beta,\sigma}((0,T];\mathbb R). $$

By Theorem \ref{T9},   $A^{\kappa_2}  W_G$ is a Gaussian process on $(0,T]$. Thanks to the estimate in Step 2,  Theorem \ref{Kol2}  applied to  $A^{\kappa_2}  W_G$ provides that  
$$A^{\kappa_2}  W_G\in \mathcal C^\gamma([\epsilon,T];E) \hspace{1cm} \text{a.s.}$$

In order to  prove that $\mathbb E \|A^{\kappa_2} W_G\|\in  \mathcal F^{\beta,\sigma}((0,T];\mathbb R), $
we again use  the estimate in  Step 2. We have   
\begin{align*}
 [\mathbb E\|A^{\kappa_2} W_G(t)-A^{\kappa_2} W_G(s)\|]^2 &\leq \mathbb E\|A^{\kappa_2} W_G(t)-A^{\kappa_2} W_G(s)\|^2\\
 &\leq  m(t)^2 s^{2(\beta-\sigma-1)} (t-s)^{2\sigma}.\end{align*}
Then,
\begin{equation*} 
\frac{ s^{1-\beta+\sigma} \mathbb E\|A^{\kappa_2} W_G(t)-A^{\kappa_2} W_G(s)\| }{ (t-s)^{\sigma}} \leq m(t).
\end{equation*} 
This imlies that 
\begin{equation}  \label{TVT46}
\sup_{0\leq s<t\leq T} \frac{ s^{1-\beta+\sigma} |\mathbb E\|A^{\kappa_2} W_G(t)\|-\mathbb E\|A^{\kappa_2} W_G(s)\| |}{ (t-s)^{\sigma}}<\infty
\end{equation}
and 
\begin{equation}  \label{TVT47}
\lim_{t\to 0} \sup_{0\leq s<t} \frac{ s^{1-\beta+\sigma} |\mathbb E\|A^{\kappa_2} W_G(t)\|-\mathbb E\|A^{\kappa_2} W_G(s)\| |}{ (t-s)^{\sigma}}=0.
\end{equation}

On the other hand, repeating the argument  as in  \eqref{TVT44} and  \eqref{TVT45}, we have 
\begin{align*}
\mathbb E\|A^{\kappa_2} W_G(t)\|^2&\leq c(E) \int_0^t \|A^{\kappa_2} S(t-s) G(s)\|_{\gamma(H;E)}^2 ds\\
&\leq C  \|A^{-\alpha_2}G\|_{\mathcal F^{\beta,\sigma}(\gamma(H;E))}^2 \max\{t^{2(\beta-\alpha_2-\kappa_2)-1}, t^{2\beta-1}\}.
\end{align*}
Thereby,
\begin{align*}
t^{1-\beta} \mathbb E\|A^{\kappa_2} W_G(t)\|&\leq t^{1-\beta} \sqrt{\mathbb E\|A^{\kappa_2} W_G(t)\|^2} \\
&\leq C  \|A^{-\alpha_2}G\|_{\mathcal F^{\beta,\sigma}(\gamma(H;E))} \max\{ t^{\frac{1}{2}-\alpha_2-\kappa_2}, t^{\frac{1}{2}}\}. 
\end{align*}
Hence,
\begin{equation}  \label{TVT48}
\lim_{t\to 0} t^{1-\beta}  \mathbb E\|A^{\kappa_2} W_G(t)\|=0.
\end{equation}

By \eqref{TVT46}, \eqref{TVT47} and \eqref{TVT48},
we conclude that 
$$\mathbb E \|A^{\kappa_2} W_G\|\in  \mathcal F^{\beta,\sigma}((0,T];\mathbb R). $$

Thanks to Steps 1 and 3,  the proof of the theorem  is now complete.
\end{proof}

\begin{theorem}   \label{T3.5}
Let {\rm (A)} and {\rm (G)} be satisfied. Assume that  $\alpha_2<\frac{-1}{2}$. Then,
$$W_G \in \mathcal C((0,T];\mathcal D(A)) \hspace{1cm} \text{ a.s.}$$
and
$$W_G(t)=-\int_0^t AW_G(s)ds +\int_0^s G(s) dW(s) \hspace{1cm} \text{ a.s., } 0<t\leq T.$$
\end{theorem}

\begin{proof}
 Theorem  \ref{T3} for $\kappa_1=1$  provides that 
$$W_G \in \mathcal C((0,T];\mathcal D(A)) \hspace{1cm} \text{ a.s.}$$
and 
$$AW_G(t)= \int_0^t AS(t-s)G(s)dW(s), \hspace{1cm} 0<t\leq T.$$

The process $\int_0^\cdot G(s)dW(s)$ is also  well-defined and continuous on $[0,T]$ because 
\begin{align*}
&\int_0^t\|G(s)\|_{\gamma(H;E)}^2 ds\\
& \leq \int_0^t\|A^{\alpha_2}\|^2 \|\|A^{-\alpha_2} G(s)\|_{\gamma(H;E)}^2 ds\\
&\leq \|A^{\alpha_2}\|^2  \|A^{-\alpha_2}G\|_{\mathcal F^{\beta,\sigma}(\gamma(H;E))}^2  \int_0^ts^{2(\beta-1)} ds\\
&=  \|A^{\alpha_2}\|^2  \|A^{-\alpha_2}G\|_{\mathcal F^{\beta,\sigma}(\gamma(H;E))}^2 \frac{t^{2\beta-1}}{2\beta-1}<\infty, \hspace{1cm} 0\leq t\leq T.
\end{align*}

Using the Fubini theorem, we have
\begin{align*}
A \int_0^t W_G(s)ds&=\int_0^t \int_0^s AS(s-u) G(u)dW(u)ds\\
&=\int_0^t \int_u^t AS(s-u) G(u)dsdW(u)\\
&=\int_0^t [G(u)-S(t-u)G(u)]dW(u)\\
&=\int_0^t G(u)dW(u)-\int_0^t S(t-u)G(u)dW(u)\\
&=\int_0^t G(u)dW(u)-W_G(t),  \hspace{2cm}  0< t \leq T.
\end{align*}
Hence,
$$W_G(t)=-\int_0^t AW_G(s)ds +\int_0^s G(s) dW(s) \hspace{1cm} \text{ a.s., } 0<t\leq T.$$
The theorem has  been thus proved.
\end{proof}

We are now ready to state the  regularity for \eqref{TVT1}.
\begin{theorem} \label{T4}
 Let {\rm (A)},  {\rm (F2)} and {\rm (G)} be satisfied. Assume that $\mathbb E \xi<\infty$. 
\begin{itemize}
  \item [\rm{(i)} ]  
 Let $-\infty< \kappa \leq 1-\alpha_1 $ and $\kappa < \frac{1}{2}-\alpha_2. $ Then, there exists a unique  mild solution of \eqref{TVT1} possessing the  regularity:
$$X\in \mathcal C((0,T];\mathcal D(A^\kappa)) \hspace{1cm} \text{a.s.} $$
with the estimate
\begin{align} 
\mathbb E&\|A^\kappa X(t)\|              \label{TVT49}\\
\leq &
C  \mathbb E\|\xi\|  t^{-\kappa}+C   \|A^{-\alpha_1}F\|_{\mathcal F^{\beta,\sigma}(E)}t^{\beta-1} + C\|A^{-\alpha_2}G\|_{\mathcal F^{\beta,\sigma}(\gamma(H;E))}  \notag \\
& \times \max\{t^{\beta-\alpha_2-\kappa-\frac{1}{2}}, t^{\beta-\frac{1}{2}}\}, \hspace{2cm} 0<t\leq T.  \notag  
\end{align}
If $\alpha_1\leq 0$ and $\alpha_2 < \frac{-1}{2}$, then $X$ becomes a strict solution of  \eqref{TVT1}. 
 \item [\rm{(ii)} ]  Assume that  $\xi\in\mathcal D(A^{\beta-\alpha_1}) $ a.s.  Let  $-\infty <\kappa  \leq  \min\{\beta-\alpha_1, \beta-\alpha_2-\frac{1}{2}\}$ and $\kappa < \frac{1}{2}-\alpha_2.$ Then,
$$X\in \mathcal C([0,T];\mathcal D(A^\kappa))\hspace{1cm} \text{a.s.} $$
with the estimate 
\begin{align*} 
&\mathbb E\|A^\kappa X(t)\|  \\
 \leq  & C[\mathbb E\|A^{\beta-\alpha_1} \xi\|+\|A^{-\alpha_1} F\|_{\mathcal F^{\beta,\sigma}(E)}  \notag\\
&+ \|A^{-\alpha_2}G\|_{\mathcal F^{\beta,\sigma}(\gamma(H;E))}      \max\{t^{\beta-\alpha_2-\kappa-\frac{1}{2}}, t^{\beta-\frac{1}{2}}\}], \hspace{0.5cm} 0< t\leq T.   \notag
\end{align*}
Furthermore, if $\kappa< \min\{\frac{1}{2}-\sigma-\alpha_2,1\}, $  then
   for any  $0<\epsilon\leq T$ and $0<\gamma <\sigma,$
$$ A^\kappa X\in \mathcal C^\gamma ([\epsilon,T];E)\hspace{1cm}\text{ a.s.}$$
and 
\begin{equation}  \label{TVT50}
\begin{cases}
\mathbb E \|A^\kappa  X\|\in  \mathcal F^{\beta,\sigma}((0,T];\mathbb R), \\
\mathbb E A^\kappa X, \frac{d\mathbb EA^\kappa X}{dt} \in   \mathcal F^{\beta,\sigma}((0,T];E).
\end{cases}
\end{equation}
 \end{itemize}
\end{theorem}

\begin{proof}
Theorems \ref{T1} and   \ref{T3} provide that \eqref{TVT1} has a unique mild solution in the space 
$$X\in \mathcal C((0,T];\mathcal D(A^\kappa)) \hspace{1cm} \text{a.s.} $$
In addition, if $\alpha_1\leq 0$ and $\alpha_2 < \frac{-1}{2}$, then by Theorems \ref{T1} and   \ref{T3.5},  $X$ becomes a strict solution.  

For Part {\rm (i)},  it now suffices to prove \eqref{TVT49}. 
Using \eqref{TVT14} and \eqref{TVT43}, we have 
\begin{align*} 
 \mathbb E\|A^\kappa X(t)\|       
= & \mathbb E\Big\|A^\kappa S(t)\xi+ \int_0^tA^\kappa S(t-s) F(s)ds+ A^\kappa W_G(t)\Big\|       \notag\\
\leq &   \mathbb E\|A^\kappa S(t)\xi\| + \int_0^t\|AS(t-s) [A^{-\alpha_1}F(s)-A^{-\alpha_1}F(t)]\|ds         \notag\\
&+ \| [I-S(t)]A^{-\alpha_1}F(t)\|  +\mathbb E\|A^\kappa W_G(t)\|     \notag\\
\leq &  \|A^\kappa S(t)\|\mathbb E \|\xi\| +\iota_1  B(\beta-\sigma,\sigma)\|A^{-\alpha_1}F\|_{\mathcal F^{\beta,\sigma}(E)} t^{\beta-1}       \notag\\
& + \| I-S(t)\| \|A^{-\alpha_1}F(t)\|  
+C\|A^{-\alpha_2}G\|_{\mathcal F^{\beta,\sigma}(\gamma(H;E))}  \notag\\
& \times \max\{t^{\beta-\alpha_2-\kappa-\frac{1}{2}}, t^{\beta-\frac{1}{2}}\}, \hspace{2cm} 0<t\leq T.        \notag
\end{align*}
Then, \eqref{TVT6},  \eqref{TVT7} and  \eqref{TVT9} gives 
\begin{align*} 
 \mathbb E & \|A^\kappa X(t)\|   \\
\leq & \iota_\kappa  \mathbb E\|\xi\|  t^{-\kappa}+[1+\iota_0+\iota_1  B(\beta-\sigma,\sigma)]   \|A^{-\alpha_1}F\|_{\mathcal F^{\beta,\sigma}(E)}t^{\beta-1}  \notag \\
&+ C\|A^{-\alpha_2}G\|_{\mathcal F^{\beta,\sigma}(\gamma(H;E))}  \max\{t^{\beta-\alpha_2-\kappa-\frac{1}{2}}, t^{\beta-\frac{1}{2}}\},\hspace{2cm} 0<t\leq T.
\end{align*}
 Thus, \eqref{TVT49} has been verified.

It is easily  seen that Part {\rm (ii)}  (except \eqref{TVT50}) 
follows from Theorems \ref{T2} and   \ref{T3} and a note that for any $0<\epsilon\leq T$ and $0<\gamma <\sigma$, 
$$\mathcal F^{\beta,\sigma}((0,T];E) \subset  \mathcal C^\gamma ([\epsilon,T];E).$$

Let us finally prove \eqref{TVT50}. We have
$$A^\kappa X=A^\kappa X_1 + A^\kappa W_G,$$
where
$$X_1=S(t)\xi+ \int_0^t S(t-s)F(s) ds.$$
In the proofs for  Theorems \ref{T2} (see Step 2) and   \ref{T3} (see \eqref{TVT48}), we already show that 
$$\lim_{t\to 0} t^{1-\beta} A^{1-\alpha_1} X_1(t)=0,$$
and
$$\lim_{t\to 0} t^{1-\beta} \mathbb E\|A^\kappa W_G(t)\|=0.$$
Since $\kappa \leq 1-\alpha_1$, we obtain that 
$$\lim_{t\to 0} t^{1-\beta} \mathbb E\|A^\kappa X(t)\|=0.$$
This means that $\mathbb E\|A^\kappa X\|$ satisfies \eqref{TVT3}.

On the other hand, by Theorem  \ref{T2},
\begin{equation}  \label{TVT51}
A^{1-\alpha_1} X_1 \in \mathcal F^{\beta,\sigma} ((0,T];E) \hspace{1cm} \text{ a.s.}
\end{equation}
Hence, it is easily seen that 
$$\mathbb E\|A^\kappa X_1\| \in \mathcal F^{\beta,\sigma} ((0,T];\mathbb R).$$
In addition, Theorem  \ref{T3} provides that 
$$\mathbb E\|A^\kappa W_G\| \in \mathcal F^{\beta,\sigma} ((0,T];\mathbb R).$$

Using the inequality:
\begin{align*}
& |\mathbb E \|A^\kappa X(t)\|-\mathbb E \|A^\kappa X(s)\| | \\
 \leq & 
|\mathbb E \|A^\kappa X_1(t)\|-\mathbb E \|A^\kappa X_1 (s)\| |   \\
&  +
|\mathbb E \|A^\kappa W_G(t)\|-\mathbb E \|A^\kappa W_G(s)\| |, \hspace{1cm} 0\leq t,s\leq T,
\end{align*}
it is easily seen that $\mathbb E\|A^\kappa X\|$ satisfies \eqref{TVT4} and \eqref{TVT5}.

In this way, we obtain that
$$\mathbb E\|A^\kappa X\| \in \mathcal F^{\beta,\sigma} ((0,T];\mathbb R).$$

 We now have
$$\mathbb E  A^\kappa X(t)=\mathbb E  A^\kappa X_1(t) =A^{\kappa+\alpha_1-1} \mathbb E  A^{1-\alpha_1} X_1(t).$$
Since $\kappa \leq 1-\alpha_1$, \eqref{TVT51} gives 
$$\mathbb E  A^\kappa X  \in \mathcal F^{\beta,\sigma} ((0,T];E).$$
In addition, since
\begin{align*}
\frac{d \mathbb E  A^\kappa X }{dt}&=\frac{d}{dt} [S(t)\mathbb E  A^\kappa \xi+\int_0^t  A^\kappa S(t-s) F(s)ds ]\\
&= -A \mathbb E  A^\kappa  X +A^\kappa F(t),
\end{align*}
we arrive at
$$A^{-1}\frac{d \mathbb E  A^\kappa X }{dt} \in  \mathcal F^{\beta,\sigma} ((0,T];E).$$
The proof is now complete.
\end{proof}

The following corollary is a direct consequence of Theorem \ref{T4}.
\begin{corollary}  \label{cor1}
Let {\rm (A)},  {\rm (F2)} and {\rm (G)} be satisfied. Assume that $\alpha_1\leq 0, \alpha_2 < \alpha_1 -\frac{1}{2}$,  and  $ \xi\in\mathcal D(A^{\beta-\alpha_1}) $ a.s.   Then,  \eqref{TVT1} possesses a unique strict solution with the  regularity: 
\begin{align*}
  X \in \mathcal C([0,T];\mathcal D(A^{\beta-\alpha_1})), \quad AX\in \mathcal C ((0,T];E)\hspace{1cm}\text{ a.s.} 
\end{align*}
\end{corollary}

\section{An application to heat equations}\label{section5}
Consider the following nonlinear stochastic heat equation:
\begin{equation}  \label{TVT53}
\begin{cases}
\frac{\partial u}{\partial t}=\Delta u-a(x) u+b(t,x)+\sigma(t,x)\frac{\partial W}{\partial t} (t,x), \hspace{0.1cm} 0< t \leq T, x\in \mathbb R^d,\\
u(0,x)=u_0(x),  \hspace{1cm}  x\in \mathbb R^d,
\end{cases}
\end{equation} 
where 
\begin{itemize}
  \item  $\Delta=\sum_{i=1}^d \frac{\partial^2}{\partial x_i^2}$ is the Laplace operator.
  \item $a(\cdot)$, $u_0(\cdot)$, and $b(\cdot,\cdot), \sigma(\cdot,\cdot)$ are real-valued functions in $\mathbb R^d $ and in $[0,T] \times \mathbb R^d$, respectively.
  \item   $\frac{\partial W (\cdot,\cdot)}{\partial t} $ is a space-time white noise with intensity $\sigma(t,x)$ at $(t,x)$. 
\end{itemize}

Let us first make precise what we mean by $ \frac{\partial W}{\partial t} (t,x)$. Since the process $W(t,x)$ depends on both position  $x$ and time $t$, it is often chosen of the form
\begin{equation*} 
W(t,x)=\sum_{j=1}^\infty e_j(x) B_j(t),
\end{equation*}
where $\{e_j\}_{j=1}^\infty$ is an orthonormal and complete basis of some Hilbert space, say  $H_0$, and $\{B_j\}_{j=1}^\infty$ is a family of independent real-valued standard Wiener processes on a filtered, complete probability space $(\Omega, \mathcal F,\{\mathcal F_t\}_{t\geq 0}, \mathbb P).$

It is known that (see, e.g., \cite{prato}) the series $\sum_{j=1}^\infty e_j B_j(t)$ converges to 
 a cylindrical Wiener process on a separable Hilbert space $H\supseteq H_0$. (The embedding of $H_0$  into $H$ is a Hilbert-Schmidt operator.)  We still denote the cylindrical Wiener process by $\{W(t), t\in [0,T]\}.$  
The noise term $\sigma(t,x)\frac{\partial W}{\partial t} (t,x)$ in \eqref{TVT53} is therefore  considered as $G(t) \frac{\partial W}{\partial t} (t)$,  where
$W$ is the cylindrical Wiener process on  $H$ and $G(t), 0\leq t \leq T,  $ are  linear operators from  $H$ to some Banach space.

We now want to consider \eqref{TVT53} in the Hilbert space $(E,\|\cdot\|)=(H^{-1}(\mathbb R^d),\|\cdot\|_{H^{-1}(\mathbb R^d)}).$  Clearly,  $E$ is a UMD Banach space of type 2 (see Remark \ref{rm0}). We assume that
\begin{itemize}
    \item  The function $F$ defined by $F(t)=b(t,\cdot)$   is an $H^{-1}(\mathbb R^d)$\,{-}\,valued measurable function on $[0,T]$.
    \item Operators $G(t), 0\leq t\leq T,$ are Hilbert-Schmidt operators from $H$ to $H^{-1}(\mathbb R^d)$ (see Remark \ref{rm1}). In addition, $G\colon [0,T] \to L_2(H; H^{-1}(\mathbb R^d))$ is $H$\,{-}\,strongly measurable and $G\in L^2((0,T); L_2(H; H^{-1}(\mathbb R^d))).$
    \item $a(\cdot) \in L^\infty (\mathbb R^d) $ with $\inf_{x\in \mathbb R^d} a(x)>0$.
\end{itemize}

Let $A$ be a realization of the differential operator 
$-\Delta +a(x)$ in $H^{-1}(\mathbb R^d).$ Thanks to \cite[Theorem 2.2]{yagi},  $A$ is a sectorial operator on $H^{-1}(\mathbb R^n)$   with domain
$\mathcal D(A)=H^1(\mathbb R^n).$
 As a consequence, $(-A)$ generates an analytical semigroup on $H^{-1}(\mathbb R^n)$.

Using $A, F$ and $G$, the equation \eqref{TVT53} is formulated as a problem of the form \eqref{TVT1} in $H^{-1}(\mathbb R^d). $  Consider separately the deterministic case and the stochastic case.

{\bf Case 1}. $\sigma(\cdot,\cdot)\equiv 0.$

Theorem \ref{T2} is available for the heat equation \eqref{TVT53} in this  case. We then obtain the maximal regularity for \eqref{TVT53}.

\begin{theorem}   \label{T5}
Assume that $F$ satisfies the condition  {\rm (F1)} with $E=H^{-1}(\mathbb R^d)$.
 Let  $u_0\in \mathcal D(A^{\beta-\alpha_1}).$  Then, \eqref{TVT53} possesses  a unique  mild solution  in the spaces: 
$$u\in \mathcal C((0,T];\mathcal D(A^{1-\alpha_1})) \cap \mathcal C([0,T];\mathcal D(A^{\beta-\alpha_1})),$$
and 
$$A^{1-\alpha_1} u\in \mathcal F^{\beta,\sigma}((0,T];H^{-1}(\mathbb R^d)). $$
  In addition, 
$u$ satisfies the estimate:
\begin{align*}
& \|A^{\beta-\alpha_1}u\|_{\mathcal C} +\|A^{1-\alpha_1} u\|_{\mathcal F^{\beta,\sigma}(H^{-1}(\mathbb R^d))} \\
& \leq  C[\|A^{\beta-\alpha_1} u_0\|       +\|A^{-\alpha_1} F\|_{\mathcal F^{\beta,\sigma}(H^{-1}(\mathbb R^d))}].   
\end{align*}
Furthermore,
when $  \alpha_1 \leq 0,$   $u$ becomes a strict solution of \eqref{TVT11} possessing the regularity: 
$$u\in \mathcal C^1((0,T];H^{-1}(\mathbb R^d))$$ 
 and 
$$A^{-\alpha_1}\frac{du}{dt}\in \mathcal F^{\beta,\sigma}((0,T];H^{-1}(\mathbb R^d))  $$
with  the estimate:  
\begin{equation*} 
\Big\|A^{-\alpha_1}\frac{du}{dt}\Big\|_{\mathcal F^{\beta,\sigma}(H^{-1}(\mathbb R^d))} \leq C [\|A^{\beta-\alpha_1} u_0\| + \|A^{-\alpha_1} F\|_{\mathcal F^{\beta,\sigma}(H^{-1}(\mathbb R^d))}]. 
\end{equation*}
Here,   $C$ is some positive constant depending only on the exponents.
\end{theorem}

{\bf Case 2}. $\sigma(\cdot,\cdot)\not\equiv 0.$ 

 Theorem \ref{T4}  is available  to the heat equation \eqref{TVT53} in this  case. The following theorem shows existence of mild and strict solutions as well as their space-time regularity to  \eqref{TVT53}.

\begin{theorem} \label{T6}
Assume that $F$ and $G$  satisfy the conditions   {\rm (F2)} and {\rm (G)}  with $E=H^{-1}(\mathbb R^d)$.
\begin{itemize}
  \item [\rm{(i)} ]  
 Let $\kappa \leq 1-\alpha_1 $ and $\kappa < \frac{1}{2}-\alpha_2. $ Then,   \eqref{TVT53} possesses a unique  mild solution  in the space:
$$u\in \mathcal C((0,T];\mathcal D(A^\kappa)) \hspace{1cm} \text{a.s.} $$
with the estimate
\begin{align*} 
\mathbb E&\|A^\kappa u(t)\|              \\
\leq &
C  \|u_0\|   t^{-\kappa}+C   \|A^{-\alpha_1}F\|_{\mathcal F^{\beta,\sigma}(E)}t^{\beta-1} + C\|A^{-\alpha_2}G\|_{\mathcal F^{\beta,\sigma}(L_2(H;H^{-1}(\mathbb R^d)))}  \notag \\
& \times \max\{t^{\beta-\alpha_2-\kappa-\frac{1}{2}}, t^{\beta-\frac{1}{2}}\}, \hspace{2cm} 0<t\leq T.  \notag  
\end{align*}
If $\alpha_1\leq 0$ and $\alpha_2 < \alpha_1-\frac{1}{2}$, then $u$ becomes a strict solution of   \eqref{TVT53}. 
 \item [\rm{(ii)} ] Take  $u_0\in\mathcal D(A^{\beta-\alpha_1}). $  Let  $-\infty <\kappa  \leq  \min\{\beta-\alpha_1, \beta-\alpha_2-\frac{1}{2}\}$ and $\kappa < \frac{1}{2}-\alpha_2.$ Then,
$$u\in \mathcal C([0,T];\mathcal D(A^\kappa))\hspace{1cm} \text{a.s.} $$
with the estimate 
\begin{align*} 
&\mathbb E\|A^\kappa u(t)\|    \\
 \leq  & C[\|A^{\beta-\alpha_1} u_0\|+\|A^{-\alpha_1} F\|_{\mathcal F^{\beta,\sigma}(H^{-1}(\mathbb R^d))}  \notag\\
&+ \|A^{-\alpha_2}G\|_{\mathcal F^{\beta,\sigma}(L_2(H;H^{-1}(\mathbb R^d)))}      \max\{t^{\beta-\alpha_2-\kappa-\frac{1}{2}}, t^{\beta-\frac{1}{2}}\}], \hspace{0.5cm} 0< t\leq T.   \notag
\end{align*}
Furthermore, if $\kappa< \min\{\frac{1}{2}-\sigma-\alpha_2,1\},$  then
   for any  $0<\epsilon\leq T$ and $0<\gamma <\sigma,$
$$ A^\kappa u\in \mathcal C^\gamma ([\epsilon,T];H^{-1}(\mathbb R^d))\hspace{1cm}\text{ a.s.}$$
and 
\begin{equation*}  
\begin{cases}
\mathbb E \|A^\kappa  u\|\in  \mathcal F^{\beta,\sigma}((0,T];\mathbb R), \\
\mathbb E A^\kappa u, \frac{d\mathbb EA^\kappa u}{dt} \in   \mathcal F^{\beta,\sigma}((0,T];H^{-1}(\mathbb R^d)).
\end{cases}
\end{equation*}
 \end{itemize}
\end{theorem}


\end{document}